\documentclass[openright,a4paper, 11pt]{article}

\usepackage[utf8]{inputenc}
\usepackage[T1, T2A]{fontenc}
\usepackage{amsmath}
\usepackage{amsthm}
\usepackage{amssymb}
\usepackage[american]{babel}
\usepackage{graphicx}
\usepackage{subcaption}
\usepackage{wrapfig}
\usepackage{colortbl}
\usepackage[usenames, dvipsnames]{xcolor}
\usepackage{tikz}
\usepackage[linktocpage=true,pagebackref, linktoc=all,hidelinks,hypertexnames=true]{hyperref}

\theoremstyle{definition} 

\newtheorem{thm}{Theorem}[section]      
\newtheorem{definition}{Definition}[section]
\newtheorem{prop}{Proposition}[section]
\newtheorem{lemma}{Lemma}[section]
\newtheorem{corollary}{Corollary}[section]
\newtheorem{problem}{Question}[section]
\newtheorem{remark}{Remark}[section]
\newtheorem{example}{Example}[section]

\newcommand{\E}{{\mathfrak E}}
\DeclareMathOperator{\trop}{trop}

\newcommand{\ZZ}{{\mathbb Z}}
\newcommand{\PP}{{\mathbb P}}
\newcommand{\KK}{{\mathbb K}}
\newcommand{\CC}{{\mathbb C}}
\newcommand{\RR}{{\mathbb R}}
\newcommand{\QQ}{{\mathbb Q}}
\newcommand{\TT}{{\mathbb T}}
\newcommand{\A}{{\mathcal A}}
\newcommand{\Log}{{\text{Log}}}
\newcommand{\val}{\mathrm{val}}
\newcommand{\Val}{\mathrm{Val}}
\newcommand{\e}{\varepsilon}

\newcommand{\Trop}{\mathrm{Trop}}
\newcommand{\ord}{\mathrm{ord}}

\begin{document}
\title{A guide to tropical modifications}
\author{Nikita Kalinin
\footnote{Guangdong Technion Israel Institute of Technology (GTIIT),
241 Daxue Road, Shantou, Guangdong Province 515603, P.R. China,  Technion-Israel Institute of Technology, Haifa, 32000, Haifa district, Israel, nikaanspb@gmail.com}}

\maketitle
\medskip
\medskip
\medskip

\begin{abstract}
This paper surveys {\it tropical modifications}, a notion that has become part of the folklore of tropical geometry. Tropical modifications are used in tropical intersection theory, tropical Hodge theory, and the study of singularities. They admit interpretations in several contexts, including hyperbolic geometry, Berkovich spaces, and non-standard analysis.

The goal of the paper is to present several points of view, give references, and illustrate the usefulness of tropical modifications. We assume that the reader has already encountered tropical modifications and wants to understand them better.

The paper also contains two new contributions: an obstruction to the realizability of non-transversal intersections and a tropical version of Weil's reciprocity law.

\smallskip
\noindent{\bf 2020 Mathematics Subject Classification.} 14T10, 14T20, 14T90, 14B10, 05C99.

\smallskip
\noindent{\bf Keywords.} Tropical geometry, Weil reciprocity, tropical modifications, singularities, Hodge theory.
\end{abstract}

\begin{flushright}
{\small\bfseries Substance is by nature prior to its modifications.}\\
{\small\bfseries ... nothing is granted in addition to the understanding,}\\
{\small\bfseries except substance and its modifications.}\\
{\small\itshape Ethics, Benedictus de Spinoza.}
\end{flushright}

\medskip
\medskip
\medskip

This paper surveys {\it tropical modifications}, a notion that has become part of the folklore of tropical geometry. Tropical modifications are used in tropical intersection theory, tropical Hodge theory, and the study of singularities. They admit interpretations in several contexts, including hyperbolic geometry, Berkovich spaces, and non-standard analysis.
     
One cannot say better than in \cite{brugalle}: ``{\it Tropical modifications ... can be seen as a refinement of the tropicalization process, and allows one to recover some information ... sensitive to higher order terms.}'' Tropical modifications were defined by Mikhalkin in \cite{mikh2}.

The goal of the paper is to present several points of view, give references, and illustrate the usefulness of tropical modifications. We assume that the reader has already encountered tropical modifications and wants to understand them better.

The paper also contains two new contributions: an obstruction to the realizability of non-transversal intersections, Theorem~\ref{th_subordinate}, and a tropical version of Weil's reciprocity law, Theorem~\ref{th_tropicalweil}. A generalization of tropical momentum is given in Section~\ref{sec_digression}. 

One of the main results is the subordination obstruction, Theorem~\ref{th_subordinate}:
if two algebraic varieties lift a prescribed tropical intersection, then the
valuation of their actual intersection divisor is not an arbitrary divisor
rationally equivalent to the stable intersection. It must be subordinate to the
stable-intersection divisor, meaning that it is cut out by a tropical function
lying below the stable-intersection function and agreeing with it on the frozen
locus. 

As a preliminary introduction to tropical geometry, see
\cite{BIMS}, \cite{2013arXiv1311.2360B} and \cite{mikhalkin2006tropical}, where
tropical modifications are also discussed. We are glad to mention other texts promoting
modifications from different perspectives: 

{\bf singularities and tangency}: \cite{brugalle} (examples, construction of curves with inflection points), \cite{2014arXiv1409.7430A} (repairing the $j$-invariant of elliptic
curves), \cite{cueto2018tropical} (tropical genus two curves from several perspectives), \cite{MarkwigRistauSchleis2023} (hyperelliptic curves) \cite{hahn2018tropicalized} (tropical quartic curves), \cite{cueto2023combinatorics} (tropical bitangent lines to a planar quartic), \cite{ganor,ganor2022enumeration} (enumeration of cuspidal curves), \cite{len2020lifting2} (lifting tropical bitangents), \cite{kalinin,mythesis, nagatakalinin} (multiplicity of tropical singular points),

{\bf realizability of intersections}:  \cite{Shaw:2015qd}  (intersection theory on tropical surfaces),  \cite{coppens2016metric} (lifting divisors, constructing curves without maps to an elliptic curve), \cite{len2020lifting} (studying tropical self-intersections),  \cite{mundinger2018image} (realizability of tropical linear subspaces),  \cite{yamamoto2018periods} (tropical $K3$ surfaces),  \cite{renaudineau2015constructions} (tropical constructions of real reducible curve in the Hirzebruch surfaces and other real algebraic varieties), \cite{bertrand2017planar} (constructing tropical curves of big genus in $\TT^n$ and other questions about realizability of big Betti numbers), \cite{he2019generalization} (lifting non-proper tropical intersections),

{\bf maps between tropical curves}: \cite{del2016tropical} (presenting metric graphs as tropical planar curves), \cite{kageyama2018divisorial,dervodeli2023construction} (gonality of tropical curves), \cite{ulirsch2019tropical}(a tropical analogue of the double ramification locus), \cite{draisma2021catalan} (maps to a tree from a modification of a genus $g$ curve), \cite{MeloZheng2026,Zakharov2025TrigonalSecondMoment} (trigonal curves),

{\bf tropical cohomology theory, Hodge theory}: \cite{jell2017poincare} (tropical Dolbeault cohomology), \cite{de2020chern} (Chern--Schwartz--MacPherson cycles of matroids), \cite{amini2021homology} (homology of tropical fans), \cite{amini2020hodge} (Hodge theory for tropical varieties).

This paper is organized as follows. Section~\ref{sec_moddefinition} defines tropical modifications via multivalued operations and discusses basic examples. We then prove several structural theorems and discuss applications. Section~\ref{sec_modmotivation} gives intuitive interpretations of tropical modifications; a reader seeking motivation may start there and then return to Section~\ref{sec_moddefinition}. The paper concludes with several questions where tropical modifications may be useful.

\subsection*{Acknowledgements}

I want to thank K. Shaw, E. Brugall\'e, and G. Mikhalkin, who taught me the concept of tropical modification, and everyone else who commented on this text.

\tableofcontents

\newpage
\section{Definitions and examples}

\subsection{Tropical modification via the graphs of multivalued functions}
\label{sec_moddefinition}

Recall that the tropical semiring $\TT$ is defined as
\[
\TT=\RR\cup\{-\infty\}, \qquad \TT^*=\TT\setminus\{-\infty\}=\RR.
\]
We extend addition from $\RR$ to $\TT$ by the rule $-\infty+A=-\infty$ for all $A\in\TT$. We extend the order relation from $\RR$ to $\TT$ by the rule $-\infty<A$ for all $A\in\RR$. A convenient way to define tropical modifications is through the multivalued tropical addition introduced by Viro in \cite{viro}.

\begin{definition} The set $\TT$ is equipped with tropical multiplication (usual addition on $\RR$), and tropical addition, defined as follows: 
\begin{itemize}  
\item the result of tropical multiplication of $A$ and $B$ is $A+B$, 
\item the result of tropical addition of $A$ and $B$ is $\max (A,B)$ if $A\ne B$, and 
\item the result of tropical addition of $A$ and $A$ is $\{x\in\TT\mid x\leq A\}$.
\end{itemize}  

We may say, equivalently, that the operation $\max$ is redefined to be multivalued in the case of equal arguments, i.e., $\max(A,A)=\{X\in\TT\mid X\leq A\}$.
\end{definition}

\begin{definition}
On the dense torus $\RR^n=(\TT^*)^n$, a tropical Laurent monomial is a function
\[
X\mapsto A+i_1X_1+\cdots+i_nX_n,\qquad A\in\RR,\quad (i_1,\dots,i_n)\in\ZZ^n.
\]
If we work on all of $\TT^n$, we restrict to ordinary tropical monomials with
\((i_1,\dots,i_n)\in\ZZ_{\ge 0}^n\). This is the tropical analogue of the ordinary monomial $aX_1^{i_1}X_2^{i_2}\cdots X_n^{i_n}$. A tropical polynomial $f$ is a {\it tropical} sum of a finite number of tropical monomials,

\begin{equation}
\label{eq_poly}
f(X_1,\dots X_n) = \max_{(i_1,\dots,i_n)\in I} (A_{i_1\dots i_n}+i_1X_1+i_2X_2+\dots+i_nX_n),
\end{equation}

where $I$ is a finite subset of $\ZZ^n$, $A_{i_1\dots i_n}\in \TT^*$, and the operation $\max$ is multivalued. 

With the above multivalued convention, a point \(X\) belongs to the zero set of \(f\) if
\(-\infty\in f(X)\). Equivalently, in the usual single-valued notation, \(X\) is a point
where the maximum in \eqref{eq_poly} is attained by at least two monomials. We say that $f$ {\it vanishes} at $X$. A tropical hypersurface (as a set) is the zero set of a tropical polynomial on $\TT^n$.
\end{definition}

\begin{remark}
Note that the zero set of $F$ is the set where the maximum in the right-hand side of \eqref{eq_poly} is attained at least twice; this coincides with the standard definition of a tropical hypersurface \cite{BIMS}. We assign weights to the maximal-dimensional faces of the tropical hypersurface. On a codimension-one face where exactly two monomials with exponent vectors
\(u,v\in\ZZ^n\) dominate, the weight is the lattice length of \(u-v\), i.e.
\[
\gcd(u_1-v_1,\dots,u_n-v_n).
\]
For non-generic faces this is understood by refinement, or equivalently by the usual lattice-index definition. Refer to \cite{mikh2} for details.

\end{remark}

\begin{definition}[Tropical modification]
\label{def_modification}
Let $N$ be a tropical hypersurface in $\TT^n$, i.e., the zero set of a tropical polynomial $f$ on $\TT^n$. The {\it modification} of $\TT^n$ {\it along} $N$ is the set 
\begin{equation}
m_N(\TT^n)=\{(X,Y)\in\TT^n\times \TT | Y\in f(X)\},
\end{equation}
 i.e., the graph of the multivalued function $f$. For a given tropical variety $K\subset \TT^n$, a tropical subvariety $K'\subset m_N(\TT^n)$ is called {\bf a} modification of $K$ if the natural projection $p:\TT^n\times \TT\to\TT^n$ restricted to $K'$ is a tropical morphism $p: K'\to K$ of degree one. We write $K'=m_N(K)$ in this case.
\end{definition}

\begin{prop}[cf. \cite{mikh2}, 1.5 B,C]
The set $m_N(\TT^n)$ coincides with the zero set of the tropical polynomial $$f'(X_1,\dots, X_n,Y) = \max(f(X_1,\dots,X_n),Y): \TT^n\times\TT \to \TT.$$
\end{prop}

\begin{definition}
For an abstract tropical variety $M$ and its subvariety $N\subset M$ defined as the zero set of a tropical function $f: M\to\TT$, we define the tropical modification $m_N(M)$ of $M$ along $N$ as the graph of $f$ in $M\times\TT$. A subvariety $K'\subset m_N(M)$ is called a modification $m_N(K)$ of $K$ along $N$ if the natural projection $K'\to K$ is a tropical morphism of degree one. 
\end{definition}

Modifications along non-principal divisors, which can be locally presented as zero sets, can be defined locally and then glued to a global abstract tropical variety in a natural way.

\subsection{Modification reveals intersections}
Consider two algebraic curves $C_1,C_2\subset (\CC^*)^2$ defined by equations $$F_1(x,y)=0,F_2(x,y)=0,$$ respectively. Let us build the map 

\begin{equation}
\label{eq_modc}
m_{C_2} :
(x,y)\to \big(x,y,F_2(x,y)\big)\in (\CC^*)^2\times \CC.
\end{equation}
 The set $m_{C_2}\big((\CC^*)^2\big)$ is the graph of $F_2$, namely $z=F_2(x,y)$.  The
intersection $C_1\cap C_2$ can be easily recovered as
$m_{C_2}(C_1)\cap \{(x,y,0)\}$. For the complex curves, this does not seem very interesting. Still, during the tropicalization process, the
plane $(x,y,0)$ goes to the plane $(X, Y,-\infty)=\{Z=-\infty\}$, and the intersection of tropical curves will be represented by certain rays going to minus infinity by $Z$ coordinate, see examples later.

\subsection{Tropical modifications as limits of amoebas} 
Look now at what happens in the limiting procedure.  Recall that a tropical curve $C\subset \TT^2$ is the tropical limit of a family 
$C_{t}\subset(\CC^*)^2, t\in \RR$ of plane algebraic curves if in the Gromov-Hausdorff sense, we have $$C = \lim_{t\to\infty}\Log_t(C_{t})$$ where we apply $\Log_t:\CC\to\TT$
coordinate-wise, i.e. in the two dimensional case $$\Log_t(C_t)=\{(\log_t|x|,\log_t|y|)|(x,y)\in C_t\}.$$ 

Let $F_{t}$ be the equation of $C_{t}$ in $(\CC^*)^2$.

\begin{prop}
Let \(F_t\) be a family of Laurent polynomials whose tropical limit is the tropical
polynomial \(F\) defining \(C\), and assume that no cancellation changes this
tropicalization. Then the tropical limit of the graphs
\[
S_t=\{(x,y,F_t(x,y))\mid (x,y)\in(\CC^*)^2\}
\subset (\CC^*)^2\times \CC
\]
is the tropical modification \(m_C\TT^2\).
\end{prop}

Near a smooth point of \(C_t\), the function \(F_t\) can be used as a local
coordinate transverse to \(C_t\). Hence the logarithm of \(F_t\) realizes
arbitrarily negative values, producing the vertical part of the tropical
modification. Therefore, after taking the logarithm near $C$, the function $\log_t |F_{t}(x,y)|$ will assume all the values of a neighborhood of $-\infty$. 

In other words, for a general point $(X,Y)\in\TT^2$ we have that $m_C(X,Y)=(X,Y,Z)$ where $Z$ is computed as $Z=\lim \log_t|F_t(x_t,y_t)|$ and $(x_t,y_t)$ are chosen such that $X=\lim\log_t|x_t|,Y=\lim\log_t|y_t|$. But if $(X, Y)\in C$, then $(x_t,y_t)$ can be chosen to be close to $C_t$, and thus $Z$ can assume any value from the interval $F(X, Y)\subset \TT$.

Use notation \eqref{eq_modc}.

\begin{prop}
Assume that \(C_{1,t}\) and \(C_{2,t}\) have tropical limits \(C_1\) and \(C_2\), that
\(C_{1,t}\) is not contained in \(C_{2,t}\), and that the projection of the limiting
graph to \(C_1\) has degree one. Then the tropical limit of
\[
m_{C_{2,t}}(C_{1,t})
\]
is a tropical modification of \(C_1\) along \(C_2\).
\end{prop}

Note that $m_{C_2}\TT^2$ depends only on $C_2$. Quite the contrary, for given tropical curves $C_1,C_2$ we {\bf can} construct different
families $C_{1,t},C_{2,t}$ and the limit $\lim_{t\to\infty}\Log_tm_{C_{2,t}}(C_{1,t})$ can be different, see numerous examples below.

\subsection{Modification and non-archimedean valuation}  
We always suppose that an algebraic hypersurface comes with a defining equation. Instead of taking the limit of amoebas we can consider non-Archimedean amoebas of the varieties defined over the valuation fields.

\begin{definition}
Let \(M'\subset(\KK^*)^n\) be a variety over a valued field \(\KK\). Let
\(N'\subset(\KK^*)^n\) be a hypersurface defined by \(f=0\). Consider the graph
\[
\Gamma_f(M')=\{(x,z)\in(\KK^*)^n\times\KK\mid x\in M',\ z=f(x)\}.
\]
The modification \(m_NM\) of \(M=\Trop(M')\) along \(N=\Trop(N')\) is
\[
\Val(\Gamma_f(M'))\subset\TT^{n+1}.
\]
\end{definition}
 
The approach with limits of amoebas gives the same results as the approach with non-Archimedean amoebas.

\begin{prop}
Consider a tropical variety $M\subset\TT^n$ and a tropical hypersurface $N$ defined by a tropical polynomial $F$. Let $\KK$ be the field of power series in $t$, converging for $t$ in a neighborhood of $0\in \CC$, and $\val:\KK^*\to \RR$ be its natural valuation (we use the convention
\[
\val\left(\sum_q a_qt^q\right)=-\min\{q:a_q\ne0\},
\]
so that \(\val(a+b)\le \max(\val(a),\val(b))\) and
\(\val(t+2t^2)=-1\)). Suppose that $$f\in \KK[x_1,\dots,x_n], f=\sum_{I\in \A} a_Ix^I \text{\ and } F = \max_{I\in\A} (\val(a_I)+I\cdot X):\TT^n\to \TT$$ where $I=(i_1,\dots,i_n)\in \ZZ_{\geq 0}^n$ are multi-indices. Let $M'\subset (\KK^*)^n$ be an affine algebraic variety, and its non-Archimedean amoeba $\Val(M')$ be $M$. For a small (by module) complex number $\e$, we can substitute $t$ as $\e$. Using this substitution we define $M'_\e\subset \CC^n$ and $f_\e\in\CC[x_1,\dots, x_n]$. Then, the three following objects coincide:
\begin{itemize}
\item the limit \[
\lim_{\varepsilon\to0}
\Log_\varepsilon\bigl(\{(x,f_\varepsilon(x))\mid x\in(\CC^*)^n\}\bigr),
\]
\item non-Archimedean amoeba $\Val(\{(x,f(x))|x\in(\KK^*)^n\})\subset \TT^{n+1}$ of the graph of $f$.
\item the tropical modification $m_N(\TT^n)$.
\end{itemize}
Additionally, two following objects coincide and equal to {\bf a} tropical modification $m_N(M)$:
\begin{itemize}
\item the limit $\lim_{\e\to 0}\Log_\e(\{x,f_\e(x)|x\in M'_\e\})$,
\item  the non-Archimedean amoeba
\[
\Val\bigl(\{(x,f(x))\mid x\in M'\}\bigr)\subset\TT^{n+1}
\]
of the graph of \(f\) on \(M'\).
\end{itemize}

\end{prop}

We repeat again that given {\bf only} tropical curves $C_1,C_2\subset \TT^2$, in general it is not possible to uniquely
``determine'' the image of $C_1$ after the modification along
$C_2$.  That is why a modification of a curve along another curve is multivalued and rather can be used as {\it a method}. 
The strategy being usually applied is the following: given two tropical curves, we lift them
in a non-Archimedean field (or present them as limits of complex curves, that is the same), then we
construct the graph of the function as above and take the non-Archimedean amoeba of the limit of amoebae. Depending on the
conditions we imposed on lifted curves (be smooth or singular, be
tangent to each other, etc.), we will have a set of possible results (which often can be described using simple combinatorial conditions) for
modification of the first curve along the second curve; see examples below.

If \(C_1\) intersects \(C_2\) transversally, then \(m_{C_2}(C_1)\) is uniquely determined. If not, there are the following conditions:
\begin{itemize} 
\item one equality (via tropical momentum or tropical Menelaus theorem):  the sum of the coordinates of
all the legs of $m_{C_2}(C_1)$ going to minus infinity by $Z$-coordinate is fixed, see Proposition~\ref{prop_brug}; 
\item one inequality (subordination of divisors): the valuation of the divisor of the intersection of lifted curves is {\it subordinate} to the stable intersection of $C_1$ and $C_2$ (Theorem~\ref{th_subordinate}). 
\end{itemize}
Both conditions have higher dimensional analogs.

\subsection{Examples}

In this section we study examples of modification, treated as
a method. The reader should not be scared by these horrific equations; they are
reverse-engineered, starting from the pictures.  All the calculations
are quite straightforward.

We start by considering the modification of a curve along itself and discuss the ambiguity that appears in this case.  Then, we consider how modifications resolve indeterminacy when the intersection of tropical
objects is non-transversal. This example promotes the point of view that a tropical modification is
the same as adding a new coordinate. So, if one changes coordinates, one can do it via repetitive modifications as in \cite{ganor2022enumeration}.

In the third example, a
modification helps recover the position of the inflection point. Also, the usefulness 
of the tropical momentum and tropical Menelaus Theorem is
demonstrated. The tropical Weil theorem that shortens the combinatorial descriptions of
possible modification results is proved in Section \ref{weil}.

In the fourth example we study the influence of a singular point on 
the Newton polygon of a curve. The identical method suits higher
dimension and different types of singularities, but more needs to be
done due to complicated combinatorics. In the same example, we describe how to find
all possible valuations of the intersections of a line with a curve,
knowing only their stable tropical intersection -- the key tool here is the
Vieta theorem. The same arguments 
may be applied for non-transversal intersections of tropical varieties
of any dimension.

\begin{example}[Modification along itself]

Consider a tropical horizontal line $L$, given by $\max(1,Y)$. This is the tropicalization of a line of the type $y=t^{-1}+o(t^{-1})$. Note that if we make a modification of a line along itself, then all its points go to the minus infinity (Figure~\ref{fig_modifitself}, left). Indeed, if $F(x,y)$ is the equation of $C$, then the set of points $\{(x,y,F(x,y))|(x,y)\in C\}$ belongs to the plane $z=0$, so $$\Val(\{(x,y,F(x,y))|(x,y)\in C\})\subset \{(X,Y,Z)\in \TT^3 | Z=-\infty\}.$$ On the other hand, if we consider two different lines $C_1,C_2$ (with equations $y=t^{-1}$ and $y=t^{-1}+t^{3}$) whose  tropicalization is $L$, then all the points in $m_{C_1}C_2$ have the valuation $-3$ of $Z$ coordinate. Again, we see an ambiguity --- even if $L$ is fixed, we can take different lifts of $L$ and have different results of the modification. On the other hand we can say that the canonical modification along itself is the result similar to Figure~\ref{fig_modifitself}, left, i.e. we might define $m_CC$ as the projection of $C$ to the plane $Z=-\infty$. Nevertheless, it is better to always keep this unambiguity in mind instead of giving a precise definition of $m_CC$. More about the question of intersections of two curves having the same tropicalization can be found in \cite{len2020lifting}.
\end{example}

\begin{figure}
\begin{center}
\begin{subfigure}[htb]{0.45\textwidth}
\begin{tikzpicture}
[y= {(0.5cm,0.5cm)}, z={(0cm,0.5cm)}, x={(0.5cm,-0.2cm)},scale=0.3]
\newcommand{\lx}{-11};
\newcommand{\rx}{4};
\newcommand{\lz}{-7};
\newcommand{\ly}{-4};
\newcommand{\ry}{4}
\newcommand{\front}{10}
\draw[black,dashed](\lx,\front,\lz)--(\lx,\front,\ry)--(\rx,\front,\ry)--(\rx,\front,\lz)--cycle;

\draw[red,very thick] (\lx,\front,\lz)--(\rx,\front,\lz);

\draw[very thick, blue] (\lx,\front,1)--(\rx,\front,1);
\newcommand{\bottom}{-14}
\draw[black,dashed](\lx,\ly,\bottom)--(\lx,1,\bottom)--(\rx,1,\bottom)--(\rx,\ly,\bottom)--cycle;
\draw[black,dashed](\lx,1,\bottom)--(\lx,\ry,\bottom)--(\rx,\ry,\bottom)--(\rx,1,\bottom)--cycle;
\draw[black,dashed](\lx,1,1)--(\lx,1,\lz)--(\rx,1,\lz)--(\rx,1,1)--cycle;
\draw[dashed, fill opacity=0.5](\lx,1,1)--(\lx,\ly,1)--(\rx,\ly,1)--(\rx,1,1)--cycle;
\draw[black,dashed](\lx,1,1)--(\lx,\ry,\ry)--(\rx,\ry,\ry)--(\rx,1,1)--cycle;
\draw[red,very thick] (\lx,1,\lz)--(\rx,1,\lz);
\draw[very thick, blue] (\lx,1,1)--(\rx,1,1);
\draw[very thick, blue] (\lx,1,\bottom)--(\rx,1,\bottom);
\end{tikzpicture}
\end{subfigure}
\begin{subfigure}[htb]{0.45\textwidth}
\begin{tikzpicture}
[y= {(0.5cm,0.5cm)}, z={(0cm,0.5cm)}, x={(0.5cm,-0.2cm)},scale=0.3]
\newcommand{\lx}{-11};
\newcommand{\rx}{4};
\newcommand{\lz}{-7};
\newcommand{\ly}{-4};
\newcommand{\ry}{4}
\newcommand{\front}{10}
\draw[black,dashed](\lx,\front,\lz)--(\lx,\front,\ry)--(\rx,\front,\ry)--(\rx,\front,\lz)--cycle;

\draw[red,very thick] (\lx,\front,-3)--(\rx,\front,-3);

\draw[very thick, blue] (\lx,\front,1)--(\rx,\front,1);
\newcommand{\bottom}{-14}
\draw[black,dashed](\lx,\ly,\bottom)--(\lx,1,\bottom)--(\rx,1,\bottom)--(\rx,\ly,\bottom)--cycle;
\draw[black,dashed](\lx,1,\bottom)--(\lx,\ry,\bottom)--(\rx,\ry,\bottom)--(\rx,1,\bottom)--cycle;
\draw[black,dashed](\lx,1,1)--(\lx,1,\lz)--(\rx,1,\lz)--(\rx,1,1)--cycle;
\draw[dashed, fill opacity=0.5](\lx,1,1)--(\lx,\ly,1)--(\rx,\ly,1)--(\rx,1,1)--cycle;
\draw[black,dashed](\lx,1,1)--(\lx,\ry,\ry)--(\rx,\ry,\ry)--(\rx,1,1)--cycle;
\draw[red,very thick] (\lx,1,-3)--(\rx,1,-3);
\draw[very thick, blue] (\lx,1,1)--(\rx,1,1);
\draw[very thick, blue] (\lx,1,\bottom)--(\rx,1,\bottom);
\end{tikzpicture}
\end{subfigure}

\end{center}
\caption{Example of a modification of a line along itself. Let $L_1,L_2$ be defined by $y=t^{-1},y=t^{-1}+t^{3}$ respectively. In each group of pictures, the bottom picture is the initial $\TT^2$, the middle picture is $m_{L_1}\TT^2$, the picture at the back is the projection to $X,Z$-plane. On the left we see the modification of $L_1$ along $L_1$, on the right we see the modification of $L_2$ along $L_1$. Red line is the result $m_{L_1}L_1$ (resp. $m_{L_1}L_2$) of the modification. }
\label{fig_modifitself}

\end{figure}
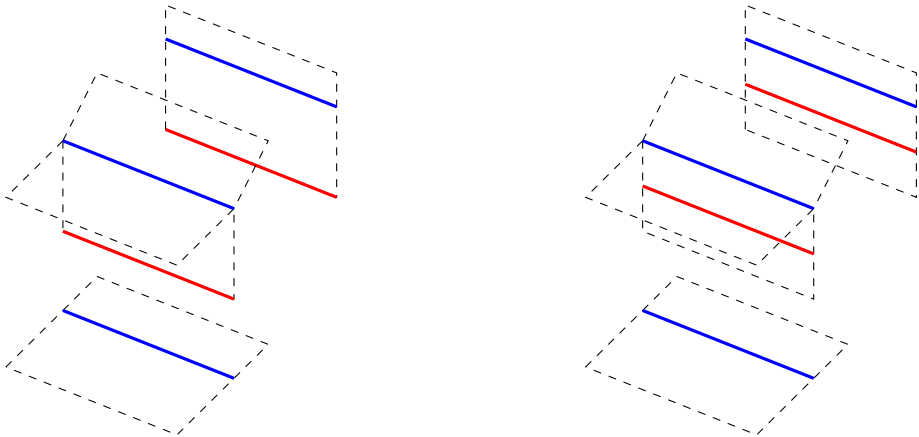

\newpage
\begin{example} [Modification, root of big multiplicity, Figure \ref{modif}]
\label{ex_big}

In this example we see two tropical curves with non-transverse
intersection which hides tangency and genus. Consider the plane curve $C$, given by the following equation:
$F(x,y)=0$, $$F(x,y) =
(x-t^{1/3})^3(x-t^{-2})+t^{-4}xy^2+(t^{-4}+2t^{-5})xy+(t^{-5}+t^{-6})x.$$
Its tropicalization\footnote{One can think that we have a family of curves $C_t$ (given by $F_t=0$)
with parameter $t$ and its tropicalization is the limit of amoebas $\lim_{t\to
  0}\Log_t(\{(x,y)|F_t(x,y)=0\})$, or that we have a curve $C$ over Puiseux
series $\CC\{\{t\}\}=\KK$ given by $\sum a_{ij}x^iy^j=0, a_{ij}\in\KK$. Its non-Archimedean amoeba is given by the set of non-smooth
points of the function $\max_{ij}(\val(a_{ij})+ix+jy)$. Both ways lead to
the same result.}
is the curve, given by the set of non-smooth points of $$\Trop(F)=\max(1, 6+x,
5+x+y,4+x+2y,5/3+2x, 2+3x,4x).$$ We
want to know what is the intersection of $C$ with the line $L$ given
by the equation $y+t^{-1}=0$. Tropicalizations of $C$ and $L$ are
drawn on Figure \ref{modif}, below, as well as the Newton polygon of $C$. The intersection is not
transverse, hence we do not know the tropicalization of $C\cap L$.
\end{example}
\begin{figure}
\begin{center}
\begin{subfigure}[htb]{0.4\textwidth}
\begin{tikzpicture}
[y= {(0.5cm,0.5cm)}, z={(0cm,0.5cm)}, x={(0.5cm,-0.2cm)},scale=0.3]
\newcommand{\lx}{-11};
\newcommand{\rx}{4};
\newcommand{\lz}{-7};
\newcommand{\ly}{-4};
\newcommand{\ry}{4}
\newcommand{\front}{10}
\draw[black,dashed](\lx,\front,\lz)--(\lx,\front,\ry)--(\rx,\front,\ry)--(\rx,\front,\lz)--cycle;

\draw[red,very thick] (-1,\front,\lz)--(-1,\front,-4)--(1,\front,0)--(2,\front,1);
\draw[red,very thick] (2,\front,1)--(2,\front,\lz);
\draw (2.7,\front, -5) node {$1$};
\draw (-0.3,\front, -5) node {$3$};

\draw[red,very thick] (-1,\front,-4)--(-5,\front,0)--(-7,\front,1);
\draw[very thick] (-7,\front,1)--(-11,\front,3);
\draw[very thick] (2,\front,1)--(4,\front,4);
\draw[red,very thick] (-5,\front,0)--(1,\front,0);
\draw[very thick, blue] (\lx,\front,1)--(\rx,\front,1);
\newcommand{\bottom}{-10}
\draw[black,dashed](\lx,\ly,\bottom)--(\lx,1,\bottom)--(\rx,1,\bottom)--(\rx,\ly,\bottom)--cycle;
\draw[black,dashed](\lx,1,\bottom)--(\lx,\ry,\bottom)--(\rx,\ry,\bottom)--(\rx,1,\bottom)--cycle;

\draw[very thick] (-7,1,\bottom)--(-7,\ly,\bottom);
\draw (-6,-3, \bottom) node {$1$};
\draw (3,-3, \bottom) node {$3$};

\draw[very thick] (2,1,\bottom)--(2,\ly,\bottom);
\draw[very thick] (-11,3,\bottom)--(-7,1,\bottom)--(2,1,\bottom)--(4,4,\bottom);
\draw[very thick, blue] (\lx,1.1,\bottom)--(\rx,1.1,\bottom);
\draw[black,dashed](\lx,1,1)--(\lx,1,\lz)--(\rx,1,\lz)--(\rx,1,1)--cycle;
\draw[dashed, fill opacity=0.5](\lx,1,1)--(\lx,\ly,1)--(\rx,\ly,1)--(\rx,1,1)--cycle;
\draw[black,dashed](\lx,1,1)--(\lx,\ry,\ry)--(\rx,\ry,\ry)--(\rx,1,1)--cycle;
\draw[red,very thick] (-1,1,\lz)--(-1,1,-4)--(1,1,0)--(2,1,1);
\draw[very thick] (2,1,1)--(4,4,4);
\draw[red,very thick] (2,1,1)--(2,1,\lz);
\draw[red,very thick] (-1,1,-4)--(-5,1,0)--(-7,1,1);
\draw[very thick] (-7,1,1)--(-11,3,3);
\draw[red, very thick] (-5,1,0)--(1,1,0);
\draw[very thick] (-7,1,1)--(-7,\ly,1);
\draw[very thick] (2,1,1)--(2,\ly,1);
\draw[very thick, blue] (\lx,1,1)--(\rx,1,1);
\end{tikzpicture}
\caption{Initial picture is below. In the center we see the limit of the graphs of the logarithm of the functions $F_{2,t}$. On the picture behind we see the projection of the graph to the plane $XZ$. Numbers on the edges are the corresponding weights.}
\label{modif}
\end{subfigure}
\quad\quad
\begin{subfigure}[htb]{0.5\textwidth}
\begin{tikzpicture}
\begin{scope}
[y= {(0.5cm,0.5cm)}, z={(0cm,0.5cm)}, x={(0.5cm,-0.2cm)},scale=0.25]
\newcommand{\lx}{-11};
\newcommand{\rx}{4};
\newcommand{\lz}{-7};
\newcommand{\ly}{-4};
\newcommand{\ry}{4}
\newcommand{\front}{10}
\draw[black,dashed](\lx,\front,\lz)--(\lx,\front,\ry)--(\rx,\front,\ry)--(\rx,\front,\lz)--cycle;
\draw[very thick, blue] (\lx,\front,1)--(\rx,\front,1);
\draw[red,very thick] (-7,\front,\lz)--(-7,\front,1);
\draw[very thick] (-7,\front,1)--(2,\front,1)--(4,\front,4);
\draw[red,very thick] (2,\front,1)--(2,\front,\lz);
\draw[very thick] (-11,\front,3)--(-7,\front,1);

\draw (-6,\front,-6) node {$1$};
\draw (3,\front,-6) node {$3$};

\newcommand{\bottom}{-10}
\draw[black,dashed](\lx,\ly,\bottom)--(\lx,1,\bottom)--(\rx,1,\bottom)--(\rx,\ly,\bottom)--cycle;
\draw[black,dashed](\lx,1,\bottom)--(\lx,\ry,\bottom)--(\rx,\ry,\bottom)--(\rx,1,\bottom)--cycle;
\draw[very thick] (-7,1,\bottom)--(-7,\ly,\bottom);
\draw[very thick] (2,1,\bottom)--(2,\ly,\bottom);
\draw[very thick] (-11,3,\bottom)--(-7,1,\bottom)--(2,1,\bottom)--(4,4,\bottom);
\draw[very thick, blue] (\lx,1.1,\bottom)--(\rx,1.1,\bottom);

\draw (-6,-3, \bottom) node {$1$};
\draw (3,-3, \bottom) node {$3$};

\draw[black,dashed](\lx,1,1)--(\lx,1,\lz)--(\rx,1,\lz)--(\rx,1,1)--cycle;
\draw[blue!100,dashed](\lx,1,1)--(\lx,\ly,1)--(\rx,\ly,1)--(\rx,1,1)--cycle;
\draw[black,dashed](\lx,1,1)--(\lx,\ry,\ry)--(\rx,\ry,\ry)--(\rx,1,1)--cycle;
\draw[very thick, blue] (\lx,1,1)--(\rx,1,1);
\draw[very thick] (-11,3,3)--(-7,1,1)--(2,1,1)--(4,4,4);
\draw[very thick] (2,1,1)--(2,1,\lz);
\draw[very thick] (-7,1,1)--(-7,1,\lz);
\draw[very thick] (-7,1,1)--(-7,\ly,1);
\draw[very thick] (2,1,1)--(2,\ly,1);
\end{scope}
\begin{scope}[shift={(2,-3)}, scale=0.4]
\draw[very thick] (0,0)--(4,0)--(1,2)--(0,0);
\draw[very thick] (1,0)--(1,2);
\draw (0,0) node {$\bullet$};
\draw (1,0) node {$\bullet$};
\draw (2,0) node {$\bullet$};
\draw (3,0) node {$\bullet$};
\draw (4,0) node {$\bullet$};
\draw (1,2) node {$\bullet$};
\draw (1,1) node {$\bullet$};
\end{scope}
\end{tikzpicture}
\caption{Notation is the same as for the picture on the left. We see the result of the modification in the case when the stable intersection is the actual intersection. The Newton polygon of the curve $C$ is depicted below.}
\label{fig_stable}
\end{subfigure}
\end{center}
\caption{Example of a modification along a line.}
\end{figure}
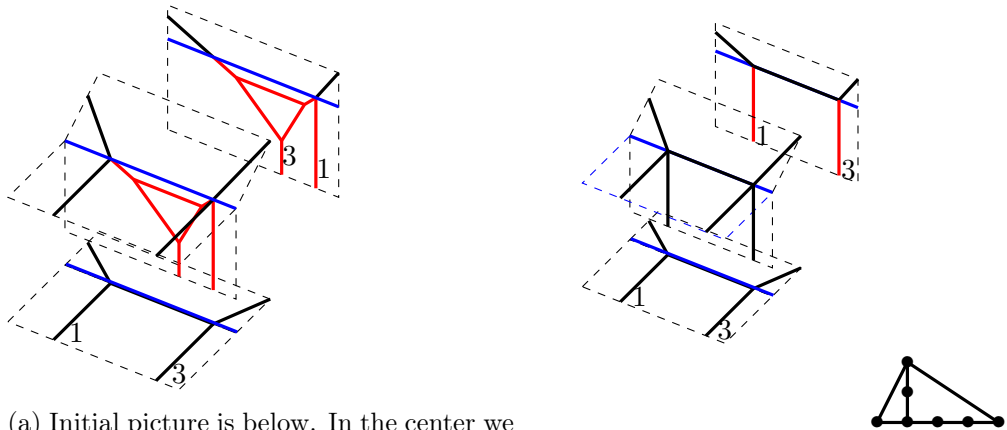

To deal with that, let us consider the map $m_{L}:(x,y)\to (x,y,y+t^{-1})$. On
Figure \ref{modif}, in the middle, we
see the tropicalization of the set $\{(x,y,y+t^{-1})\}$ and the
tropicalization of the image of $C$ under the map $m_{L}$.
Let $G(x,z)$ be the equation of the projection of $m_L(C)$ on the $xz$-plane. So, $F(x,y)=0$ implies that for the new coordinate $z=y+t^{-1}$ we have

\begin{equation}
G(x,z)=0, G(x,z)=(x-t^{1/3})^3(x-t^{-2})+t^{-4}xz+t^{-4}xz^2.
\end{equation}

Therefore the curve $C'=pr_{xz}m_L(C)$ is given by the set of non-smooth points of $$\max(1, 4+X+Y,
4+X+2Y,2+3X,4X),$$ we see $C'$ on the projection onto the plane $XZ$ on the left
part of Figure \ref{modif}. Notice that in order to recover the transversal intersection
of non-Archimedean amoebas we did nothing else
as a change of coordinates.

\begin{remark}
\label{rem_restriction}
Consider the restriction of $\Trop(F)$ on the line $Y=1$. We obtain $\max(1,7+X,5/3+2X,2+3X,4X)=\max(1,7+X,4X)$, whose locus of non-linearity corresponds to the stable intersection of our tropical curves. On the other hand, if we restrict $F$ on the line $y+t^{-1}=0$ and only then take the valuation, we obtain $\max(1,3X+2,4X)$ because $F(x,-t^{-1})=(x-t^{1/3})^3(x-t^{-2})$, and we see that this agrees with the picture of the modifications.
\end{remark}

\begin{definition}
As we see in this example, a tropical curve in $\TT^n$ typically contains infinite edges. We call them {\it legs} of a tropical curve. For each leg we have a canonical parametrization $(a_0+p_0s,a_1+p_1s,a_2+p_2s)$ where $a_i\in \RR,p_i\in\ZZ,s\in\RR, s\geq 0$, where the vector $(p_0,p_1,p_2)$, the {\it direction} of the leg, is primitive.
\end{definition}

Now, on the tropicalization of $C'$ we see a vertical leg of weight \(3\), whose \(Z\)-coordinate goes to \(-\infty\). That happens because we have the tangency
of order $3$ between $C$ and $L$, and $z$ as a function of $x$ has a root of order 3.

Note that this leg cannot mean that the point is a singular point of $C$, because the
curve $C$ (according to
criteria of \cite{markwig} or, more generally \cite{kalinin}) has no
singular points, even though the tropicalization of $C$ has an edge of
multiplicity 3. 

Thus, this new tropicalization restores the multiplicity of the intersection.
We see that the modification of the plane (i.e. amoeba of the
set $\{(x,y,y+t^{-1})\}$) is defined, but in codimension
one this procedure shows multiplicities of roots and more unapparent structures such as hidden
genus squashed initially onto an edge. One can think that this cycle was close to intersection, but after a change of
coordinates it becomes visible on the picture of the amoeba of $C'$. More examples of this kind can be found in \cite{
2014arXiv1409.7430A,cueto2018tropical}.

\begin{remark}
Nevertheless, for a general choice of representatives in Puiseaux series for these two
tropical curves $\Trop(C),\Trop(L)$, after modification we will have Figure~\ref{fig_stable},
which represents stable intersection of the curves.
\end{remark}

\begin{example}[Modification, inflection point, momentum map]
\label{ex_infl}

We consider a curve and its tangent line at an inflection
point. Suppose, that the intersection of their tropicalizations is not
transverse. How can we recover the presence of the inflection point?
\end{example}

We consider a curve $C$ with the equation $F(x,y)=0$ where
\begin{align*}F(x,y)=&y+t^{-3}xy+
(t^{-1}+4+6t+4t^2+t^3)x^2+(-t^{-3}-3-t-t^2)xy^2\\
&+(t^{-2}-t^{-1}-2+t^2+t^3)x^2y+x^2y^2,
\end{align*}
and a line $L$ with the equation $y=1+tx$. The equation of the curve is
chosen just in such a way that the restriction of $F$ on the line $L$ is
$t^2(x-1)^3(x-t^{-1})$, i.e. the point $(1,1+t)$
is the inflection point of the curve and $L$ is tangent to $C$ at this point.

Tropicalization of the curve is given by the following equation:

\begin{equation}
\Trop(F)=\max(y,x+y+3,2x+1,2x+y+2,x+2y+3,2x+2y).
\end{equation}

On Figure~\ref{fig_inflection} we see the non-Archimedean amoeba of the image of the curve under the map $(x,y)\to(x,y,y - 1- tx)$.
 
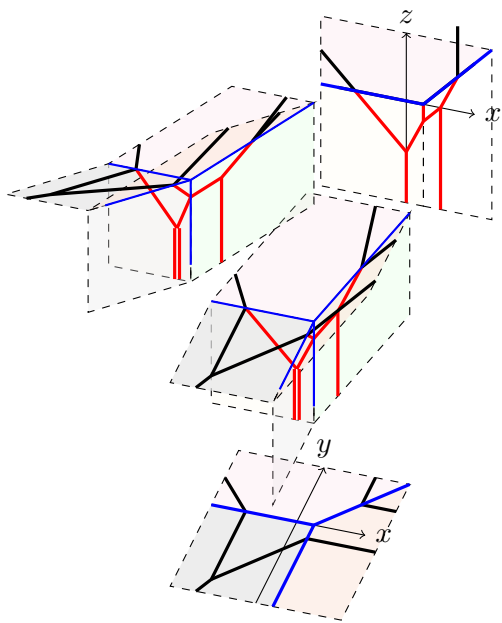
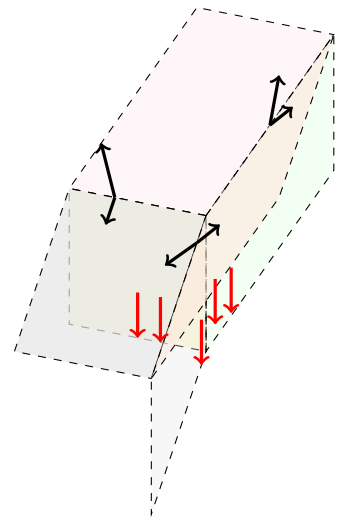
\begin{figure}[htbp]

\begin{center}
\begin{subfigure}[b]{0.55\textwidth}
\begin{tikzpicture}
[x={(0.5cm,-0.1cm)},y= {(0.2cm,0.4cm)}, z={(0.0cm,0.5cm)}, scale=0.45]

\begin{scope}[xshift=-6cm]
\newcommand{\y}{16}
\filldraw[black,fill=yellow!20,dashed, fill
opacity=0.1](1,\y,0)--(-5,\y,0)--(-5,\y,-6)--(1,\y,-6)--cycle;
\filldraw[black,dashed,fill
opacity=0.0](1,\y,0)--(5,\y,4)--(5,\y,-6)--(1,\y,-6)--cycle;
\filldraw[black,fill=Magenta!20,dashed, fill
opacity=0.2](1,\y,0)--(-5,\y,0)--(-5,\y,4)--(5,\y,4)--cycle;

\draw[->] (-4,\y,0) -- (4,\y,0);
\draw (4,\y,0) node[right] {$x$};
\draw [->] (0,\y,-6) -- (0,\y,4);
\draw (0,\y,4) node[above] {$z$};

\draw[very thick](-5,\y,2)--(-3,\y,0);
\draw[very thick](-4,\y,0)--(1,\y,0)--(4,\y,3);
\draw[very thick](3,\y,5)--(3,\y,2);
\draw[red, very thick](-3,\y,0)--(0,\y,-3)--(0,\y,-6);
\draw[red, very thick](0,\y,-3)--(1,\y,-1)--(2,\y,0)--(3,\y,2);
\draw[red, very thick](1,\y,0)--(1,\y,-1);
\draw[red, very thick](2,\y,0)--(2,\y,-6);
\draw[blue, very thick](-5,\y,0)--(1,\y,0)--(5,\y,4);

\filldraw[black,fill=yellow!20,dashed, fill
opacity=0.1](1,0,0)--(-5,0,0)--(-5,0,-6)--(1,0,-6)--cycle;
\filldraw[black,fill=Magenta!20,dashed, fill
opacity=0.2](1,0,0)--(-5,0,0)--(-1,4,4)--(5,4,4)--cycle;

\filldraw[black,fill=gray!20,dashed, fill
opacity=0.7](1,0,0)--(-5,0,0)--(-5,-6,0)--(1,-6,0)--cycle;

\filldraw[black,fill=gray!20,dashed, fill
opacity=0.3](1,0,0)--(1,-6,0)--(1,-6,-6)--(1,0,-6)--cycle;

\filldraw[black,fill=green!20,dashed, fill
opacity=0.2](1,0,0)--(5,4,4)--(5,4,-2)--(1,0,-6)--cycle;

\filldraw[black,fill=Melon!20,dashed, fill
opacity=0.6](1,0,0)--(5,4,4)--(5,-2,4)--(1,-6,0)--cycle;

\draw[very thick] (-4,1,1)--(-3,0,0);
\draw[red, very thick](-3,0,0)--(0,0,-3)--(0,0,-6);
\draw[red, double, very thick](0,0.001,-3)--(0,0.001,-6);

\draw[red, very thick] (0,0,-3)--(1,0,-1);
\draw[red, very thick] (1,0,-1)--(1,-1,0);

\draw[red, very thick] (1,0,-1)--(2,1,0)--(3,2,2);

\draw[very thick] (1,-1,0)--(-3,-5,0)--(-3.5,-6,0);
\draw[very thick](-3,-5,0)--(-3,0,0);
\draw[red, very thick](2,1,0)--(2,1,-5);
\draw[very thick] (3,2,2)--(5,2,4);
\draw[very thick] (1,-1,0)--(5,-1,4);
\draw[very thick](3,2,2)--(3,4,4);

\draw[blue, thick] (1,-5,0)--(1,0,0)--(5,4,4);
\draw[blue,thick] (-5,0,0)--(1,0,0)--(1,0,-5);

\newcommand{\z}{-12}
\filldraw[black,fill=gray!20,dashed, fill
opacity=0.7](1,0,\z)--(-5,0,\z)--(-5,-6,\z)--(1,-6,\z)--cycle;
\filldraw[black,fill=Melon!20,dashed, fill
opacity=0.6](1,0,\z)--(5,4,\z)--(5,-6,\z)--(1,-6,\z)--cycle;
\filldraw[black,fill=Magenta!20,dashed, fill
opacity=0.2](1,0,\z)--(-5,0,\z)--(-5,4,\z)--(5,4,
\z)--cycle;
\draw[->] (-4,0,\z) -- (4,0,\z);
\draw (4,0,\z) node[right] {$x$};
\draw [->] (0,-6,\z) -- (0,4,\z);
\draw (0,4,\z) node[above] {$y$};

\draw[black,thin] (-5,0,\z)--(1,0,\z)--(1,-6,\z);
\draw[black,thin] (1,0,\z)--(5,4,\z);

\draw[very thick]
(-5,2,\z)--(-3,0,\z)--(-3,-5,\z)--(-3.5,-6,\z);
\draw[very thick](-3,-5,\z)--(1,-1,\z)--(5,-1,\z);
\draw[very thick](3,4,\z)--(3,2,\z)--(5,2,\z);
\draw[blue, very thick](-5,0,\z)--(1,0,\z)--(1,-6,\z);
\draw[blue, very thick](1,0,\z)--(5,4,\z);
\end{scope}

\begin{scope}[x={(0.4cm,-0.08cm)},y= {(0.5cm,0.15cm)}, z={(0.0cm,0.5cm)}, scale=1, shift={(-35,9)}]

\filldraw[black,fill=yellow!20,dashed, fill
opacity=0.1](1,0,0)--(-5,0,0)--(-5,0,-6)--(1,0,-6)--cycle;
\filldraw[black,fill=Magenta!20,dashed, fill
opacity=0.2](1,0,0)--(-5,0,0)--(-1,4,4)--(5,4,4)--cycle;

\filldraw[black,fill=gray!20,dashed, fill
opacity=0.7](1,0,0)--(-5,0,0)--(-5,-6,0)--(1,-6,0)--cycle;

\filldraw[black,fill=gray!20,dashed, fill
opacity=0.3](1,0,0)--(1,-6,0)--(1,-6,-6)--(1,0,-6)--cycle;

\filldraw[black,fill=green!20,dashed, fill
opacity=0.2](1,0,0)--(5,4,4)--(5,4,-2)--(1,0,-6)--cycle;

\filldraw[black,fill=Melon!20,dashed, fill
opacity=0.6](1,0,0)--(5,4,4)--(5,-2,4)--(1,-6,0)--cycle;

\draw[very thick] (-4,1,1)--(-3,0,0);
\draw[red, very thick](-3,0,0)--(0,0,-3)--(0,0,-6);
\draw[red, double, very thick](0,0.001,-3)--(0,0.001,-6);


\draw[red, very thick] (0,0,-3)--(1,0,-1);
\draw[red, very thick] (1,0,-1)--(1,-1,0);

\draw[red, very thick] (1,0,-1)--(2,1,0)--(3,2,2);

\draw[very thick] (1,-1,0)--(-3,-5,0)--(-3.5,-6,0);
\draw[very thick](-3,-5,0)--(-3,0,0);
\draw[red, very thick](2,1,0)--(2,1,-5);
\draw[very thick] (3,2,2)--(5,2,4);
\draw[very thick] (1,-1,0)--(5,-1,4);
\draw[very thick](3,2,2)--(3,4,4);

\draw[blue, thick] (1,-5,0)--(1,0,0)--(5,4,4);
\draw[blue,thick] (-5,0,0)--(1,0,0)--(1,0,-5);

\end{scope}
\end{tikzpicture}
\caption{In the center we see a modification of the picture below, its $XZ$-projection is on the right, on the left we see it from a different perspective. Doubled red lines graphically represents a line with weight three.}
\label{fig_inflection}
\end{subfigure}
\quad\quad
\begin{subfigure}[b]{0.35\textwidth}
\centering
\begin{tikzpicture}
[x={(0.5cm,-0.1cm)},y= {(0.2cm,0.6cm)}, z={(0.0cm,0.5cm)}, scale=0.6,shift={(3,0)}]
\filldraw[black,fill=yellow!20,dashed, fill
opacity=0.9](1,0,0)--(-5,0,0)--(-5,0,-6)--(1,0,-6)--cycle;
\filldraw[black,fill=Magenta!20,dashed, fill
opacity=0.2](1,0,0)--(-5,0,0)--(-1,4,4)--(5,4,4)--cycle;
\filldraw[black,fill=gray!20,dashed, fill
opacity=0.7](1,0,0)--(-5,0,0)--(-5,-6,0)--(1,-6,0)--cycle;
\filldraw[black,fill=gray!20,dashed, fill
opacity=0.3](1,0,0)--(1,-6,0)--(1,-6,-6)--(1,0,-6)--cycle;
\filldraw[black,fill=green!20,dashed, fill
opacity=0.2](1,0,0)--(5,4,4)--(5,4,-2)--(1,0,-6)--cycle;
\filldraw[black,fill=Melon!20,dashed, fill
opacity=0.6](1,0,0)--(5,4,4)--(5,-2,4)--(1,-6,0)--cycle;
\draw[->,black, very thick] (-3,0,0)--(-4,1,1);

\draw[->,black, very thick] (1,-1,0)--(0,-2,0);
\draw[->,black, very thick](-3,0,0)--(-3,-1,0);
\draw[->,black, very thick] (3,2,2)--(4,2,3);
\draw[->,black, very thick] (1,-1,0)--(2,-1,1);
\draw[->,black, very thick](3,2,2)--(3,3,3);

\draw[->,red, very thick](-2,0,-4)--(-2,0,-6);
\draw[->,red, very thick](-1,0,-4)--(-1,0,-6);
\draw[->,red, very thick](1,-0.5,-4)--(1,-0.5,-6);
\draw[->,red, very thick](1,1,-4)--(1,1,-6);
\draw[->,red, very thick](1.5,1.5,-4)--(1.5,1.5,-6);
\end{tikzpicture}
\caption{Application of the generalized tropical Menelaus Theorem: we know the direction of the infinite black rays
  emanating from the tropical curve (in the center on the left), therefore an application of this theorem gives the sum of
  $X$- and $Y$- coordinates of red legs, going vertically to the bottom (these
  legs present exactly the intersection of two considered curves.)}
\label{moment}
\end{subfigure}
\caption{Example of modification in the case of inflection point. The point $(0,0)$ on the bottom picture is the
  tropicalization of the inflection point. We modified the black curve along the blue curve, red parts are the parts becoming visible after the modification.}
\end{center}
\end{figure}

In order to find  $X$-coordinates of the possible legs we can apply
the tropical momentum: see Figure~\ref{moment}.

\begin{definition}
The momentum of a leg $(A_0+P_0s,A_1+P_1s,A_2+P_2s)$ with respect to a point $(B_0,B_1,B_2)$ is the vector product $(A_0-B_0,A_1-B_1,A_2-B_2)\times (P_0,P_1,P_2)$.
\end{definition}

We will prove a (simple) theorem that {\it the sum of the moments of the legs, counted with their weights, is zero.}
Note, that in our case, all the legs we do not know are of the form $(X_0,Y_0,Z_0-s)$, because they are vertical.
Refer to Figure~\ref{moment}. So, we take the vertex $O$ of
the tropical plane, and sum up the vector products $OX_i\times X_iY_i$
where $X_iY_i$ are black legs (that we already know) and red legs (which are all vertical). Computation gives us 
\begin{align*}
(&-4,0,0)\times
(-1,1,1)+(-4,0,0)\times (0,-1,0)+(0,-1,0)\times (-1,-1,0)\\
&+(0,-1,0)\times (1,0,1)+(2,2,2)\times (1,0,1)+(2,2,2)\times (0,1,1)\\
&+(X,0,0)\times (0,0,-1)+(0,Y,0)\times (0,0,-1) + (Z+1,Z,0)\times (0,0,-1)=0,
\end{align*}
 i.e.  $(1,-2,0)+(Y+Z+1, X+Z,0)=0$, where $X$ stands for the sum of the $X$-coordinates of the vertical legs situated under the line $(1-s,0,0)$, $Y$ stands for the sum of the $Y$-coordinates of the vertical legs under the line $(1,-s,0)$, $Z$ stands for the sum of the $Y$-coordinates of the vertical legs under the line $(1,s,s)$.

On the left picture we see where the red legs are situated. 
But, since modification of a tropical curve $C$ along a tropical
curve $C'$ is not canonically defined\footnote{If the intersection $C\cap C'$ is
  transverse, then the modification is uniquely defined.}, then, for example, a modification of $C$ could differ
from $C$ just by adding vertical legs at four vertices of the
$C$: this would correspond to the stable intersection (which is
always realizable in the sense that there exists a curve in Puiseux
series, such that the valuation of their intersection is the stable intersection)

\example{Singular point, its unique position, and possible liftings of intersection.}
\label{ex_singularexample}
Consider a curve $C'$ defined by the equation $G(x,y)=0$, where 

\begin{align*}G(x,y)&=t^{-3} xy^3 -
(3t^{-3}+t^{-2})xy^2+ (3t^{-3}+2t^{-2}-2t^{-1})xy \\
&-(t^{-3}+t^{-2}-2t^{-1}-3t^2)x+t^{-2}x^2y^2
-(2t^{-2}-t^{-1})x^2y\\
&+(t^{-2}-t^{-1}-3t^2)x^2+t^{-1}y -(t^{-1}+t^2)
+t^2x^3.\end{align*}
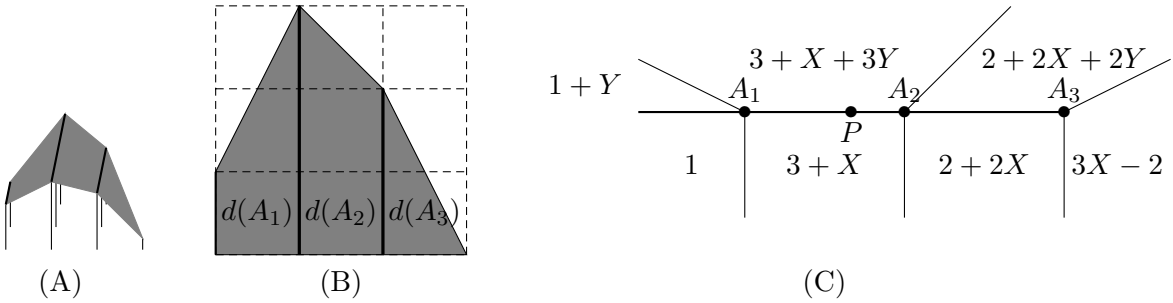
\begin{figure} [htbp]
\begin{tikzpicture}
[y= {(0.2cm,1cm)}, z={(0cm,0.5cm)}, x={(2cm,0cm)},scale=0.3]
\draw (0,0,1)--(0,0,-3);
\draw (0,1,1)--(0,1,-3);
\draw (1,3,3)--(1,3,-3);
\draw (1,2,3)--(1,2,-3);
\draw (1,1,3)--(1,1,-3);
\draw (1,0,3)--(1,0,-3);
\draw (2,0,2)--(2,0,-3);
\draw (2,1,2)--(2,1,-3);
\draw (2,2,2)--(2,2,-3);
\draw (3,0,-2)--(3,0,-3);

\filldraw[gray](0,0,1)--(1,0,3)--(1,3,3)--(0,1,1)--cycle;
\filldraw[gray](2,0,2)--(2,2,2)--(1,3,3)--(1,0,3)--cycle;
\filldraw[gray](2,2,2)--(2,0,2)--(3,0,-2)--cycle;
\draw[thick](2,2,2)--(2,0,2);
\draw[thick](1,3,3)--(1,0,3);
\draw[thick](0,0,1)--(0,1,1);
\draw (1.5,-3) node{$\mathrm{(A)}$}; 
\end{tikzpicture}
\qquad
\begin{tikzpicture}[scale=1.1]
\draw[fill=gray](0,0)--++(0,1)--++(1,2)--++(1,-1)--++(1,-2)--cycle;
\draw[very thick](1,0)--++(0,3);
\draw[very thick](2,0)--++(0,2);
\draw[thick](0,0)--++(0,1);
\draw(0.5,0.5) node {$d(A_1)$};
\draw(1.5,0.5) node {$d(A_2)$};
\draw(2.5,0.5) node {$d(A_3)$};

\draw[densely dashed] (0,0) grid (3,3);
\draw (1.5,-0.35) node{$\mathrm{(B)}$}; 
\end{tikzpicture}
\qquad
\begin{tikzpicture}[scale=0.7]

\draw[thick](3,6.5)--++(8,0);
\draw(3,7.5)--++(2,-1)--++(0,-2);
\draw(8,4.5)--++(0,2)--++(2,2);
\draw(11,4.5)--++(0,2)--++(2,1); 
\draw (5,6.5) node {$\bullet$} ;
\draw (5,6.5) node[above] {$A_1$} ;
\draw (7,6.5) node {$\bullet$} ;
\draw (7,6.5) node[below] {$P$} ;
\draw (8,6.5) node {$\bullet$} ;
\draw (8,6.5) node[above] {$A_2$} ;
\draw (11,6.5) node {$\bullet$} ;
\draw (11,6.5) node[above] {$A_3$} ;

\draw (4,5.5) node {$1$};
\draw (2,7) node {$1+Y$};
\draw (6.5,7.5) node {$3+X+3Y$};
\draw (6.5,5.5) node {$3+X$};
\draw (9.5,5.5) node {$2+2X$};
\draw (12,5.5) node {$3X-2$};
\draw (11,7.5) node {$2+2X+2Y$};
\draw (6.5,3.25) node{$\mathrm{(C)}$}; 
\end{tikzpicture}
\qquad
\begin{center}
\caption {The extended Newton
  polyhedron $\widetilde{\A}$ of the  curve $C'$ is drawn in $\mathrm{(A)}$. The projection of its faces gives us the subdivision of the
  Newton polygon of $C'$; see $\mathrm{(B)}$. The tropical curve $\Trop(C')$ is drawn
  in $\mathrm{(C)}$. The vertices $A_1,A_2,A_3$ have coordinates
  $(-2,0),(1,0),(4,0)$. The edge $A_1A_2$ has weight $3$, while
  the edge $A_2A_3$ has
  weight $2$. The point $P$ is $(0,0)=\Val((1,1))$. 
}
\label{fig_example3}
\end{center}
\end{figure}

Let us make the modification along the line $y=1$. For that we draw the graph of the function $z(x,y)=y-1$.

Note that we can easily find the number (with multiplicities) of the vertical legs. Indeed, each edge from $A_1,A_2,A_3$ going up in direction $(i,j)$ becomes after the modification a ray going in the direction $(i,j,j)$. Therefore, the total momentum of the vertical legs is the sum of $Y$-parts of momenta of the edges going up from $A_1,A_2,A_3$, that is, $3$. 
Then, if we know that after the modification our curve has a leg of weight $3$, then its unique position can be found from the generalized tropical Menelaus theorem. So, in this case (the points $(1,1)$ is of multiplicity $3$ for the curve) the pictures after the modifications is as on Figure~\ref{fig_oldexample}, left. If $\Val(C')=C$, but we do not have the other restricting condition, then the picture after the modification can be as in Figure~\ref{fig_oldexample}, right top, or right bottom, both cases can be realized.

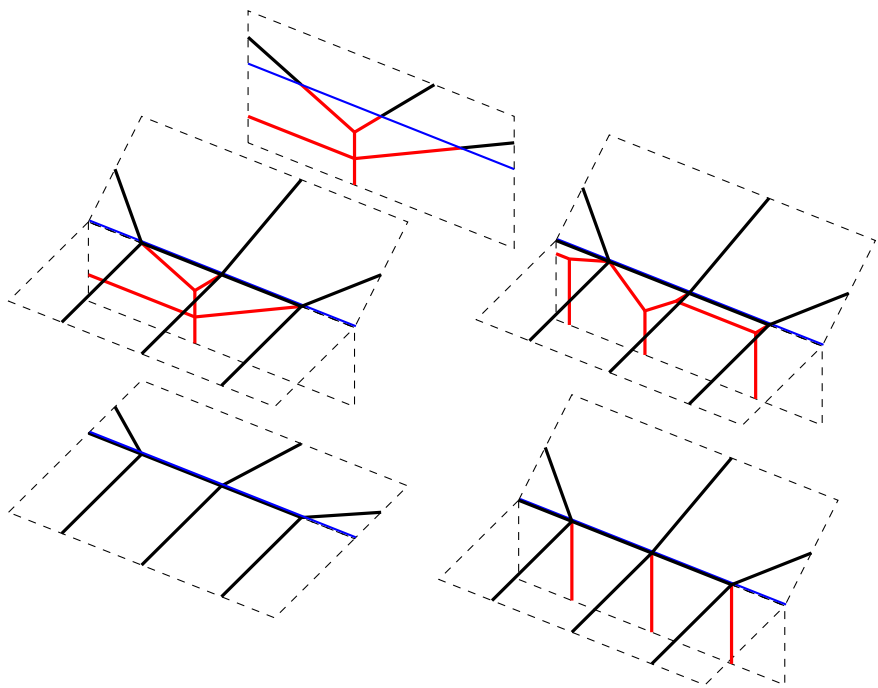
\begin{figure}[htbp]
\begin{tikzpicture}
[y= {(0.5cm,0.5cm)}, z={(0cm,0.5cm)}, x={(0.5cm,-0.2cm)},scale=0.7]
\newcommand{\lx}{-4};
\newcommand{\rx}{6};
\newcommand{\lz}{-3};
\newcommand{\rz}{2};
\newcommand{\ly}{-3};
\newcommand{\ry}{2}
\newcommand{\front}{6}
\newcommand{\bottom}{-8}
\draw[black,dashed](\lx,\front,\lz)--(\lx,\front,\rz)--(\rx,\front,\rz)--(\rx,\front,\lz)--cycle;

\draw[very thick] (-4,\front,1)--(-2,\front,0);
\draw[very thick, red] (-2,\front,0)--(0,\front,-1)--(1,\front,0);
\draw[very thick] (6,\front,1)--(4,\front,0);
\draw[very thick] (3,\front,2)--(1,\front,0);
\draw[very thick, red](4,\front,0)--(0,\front,-2)--(-4,\front,-2);
\draw[very thick, red] (0,\front,-1)--(0,\front,-3);
\draw[blue, thick] (-4,\front,0)--(6,\front,0);


\draw[black,dashed](\lx,\ly,\bottom)--(\lx,0,\bottom)--(\rx,0,\bottom)--(\rx,\ly,\bottom)--cycle;

\draw[black,dashed](\lx,0,\bottom)--(\lx,\ry,\bottom)--(\rx,\ry,\bottom)--(\rx,0,\bottom)--cycle;

\draw[very thick] (-4,1,\bottom)--(-2,0,\bottom)--(4,0,\bottom)--(6,1,\bottom);
\draw[very thick] (-4,0,\bottom)--(-2,0,\bottom)--(-2,-3,\bottom);
\draw[very thick] (2,2,\bottom)--(1,0,\bottom)--(1,-3,\bottom);
\draw[very thick] (4,0,\bottom)--(4,-3,\bottom);
\draw[blue, thick](-4,0.05,\bottom)--(6,0.05,\bottom);
\draw[black,dashed](\lx,0,0)--(\lx,0,\lz)--(\rx,0,\lz)--(\rx,0,0)--cycle;
\draw[dashed](\lx,0,0)--(\lx,\ly,0)--(\rx,\ly,0)--(\rx,0,0)--cycle;
\draw[black,dashed](\lx,0,0)--(\lx,\ry,\ry)--(\rx,\ry,\ry)--(\rx,0,0)--cycle;
\draw[red,very thick] (-2,0,0)--(0,0,-1)--(0,0,-3);
\draw[red,very thick] (4,0,0)--(0,0,-2)--(-4,0,-2);
\draw[red,very thick] (1,0,0)--(0,0,-1);
\draw[blue, thick] (-4,0.05,0)--(6,0.05,0);
\draw[very thick] (-2,0,0)--(4,0,0)--(6,1,1);
\draw[very thick] (-4,1,1)--(-2,0,0)--(-2,-3,0);
\draw[very thick] (2,2,2)--(1,0,0)--(1,-3,0);
\draw[very thick] (4,0,0)--(4,-3,0);

\begin{scope}[xshift=230,yshift=-150]
\draw[black,dashed](\lx,0,0)--(\lx,0,\lz)--(\rx,0,\lz)--(\rx,0,0)--cycle;
\draw[dashed](\lx,0,0)--(\lx,\ly,0)--(\rx,\ly,0)--(\rx,0,0)--cycle;
\draw[black,dashed](\lx,0,0)--(\lx,\ry,\ry)--(\rx,\ry,\ry)--(\rx,0,0)--cycle;
\draw[red,very thick] (-2,0,0)--(-2,0,-3);
\draw[red,very thick] (4,0,0)--(4,0,-3);
\draw[red,very thick] (1,0,0)--(1,0,-3);
\draw[blue, thick] (-4,0.05,0)--(6,0.05,0);
\draw[very thick] (-4,0,0)--(4,0,0)--(6,1,1);
\draw[very thick] (-4,1,1)--(-2,0,0)--(-2,-3,0);
\draw[very thick] (2,2,2)--(1,0,0)--(1,-3,0);
\draw[very thick] (4,0,0)--(4,-3,0);
\end{scope}

\begin{scope}[xshift=250,yshift=-10]
\draw[black,dashed](\lx,0,0)--(\lx,0,\lz)--(\rx,0,\lz)--(\rx,0,0)--cycle;
\draw[dashed](\lx,0,0)--(\lx,\ly,0)--(\rx,\ly,0)--(\rx,0,0)--cycle;
\draw[black,dashed](\lx,0,0)--(\lx,\ry,\ry)--(\rx,\ry,\ry)--(\rx,0,0)--cycle;
\draw[red,very thick] (-2,0,0)--(-3.5,0,-0.5)--(-4,0,-0.5);
\draw[red,very thick] (-3.5,0,-0.5)--(-3.5,0,-3);
\draw[red,very thick] (4,0,0)--(3.5,0,-0.5)--(3.5,0,-3);
\draw[red,very thick] (3.5,0,-0.5)--(0.5,0,-0.5)--(1,0,0);
\draw[red,very thick](0.5,0,-0.5)--(-0.66,0,-1.33)--(-2,0,0);
\draw[red,very thick](-0.66,0,-1.33)--(-0.66,0,-3);
\draw[blue, thick] (-4,0.05,0)--(6,0.05,0);
\draw[very thick] (-4,0,0)--(4,0,0)--(6,1,1);
\draw[very thick] (-4,1,1)--(-2,0,0)--(-2,-3,0);
\draw[very thick] (2,2,2)--(1,0,0)--(1,-3,0);
\draw[very thick] (4,0,0)--(4,-3,0);
\end{scope}
\end{tikzpicture}
\caption{Refer to Example~\ref{ex_singularexample}. Left bottom picture represents a curve $C$. On top of it, a modification of it is depicted, with the projection of the latter on the $XZ$-plane. On the right side we see two other possible modification of $C$.}
\label{fig_oldexample}
\end{figure}

\section{Structural theorems about tropical modification}

\subsection{Multiplicity of an intersection}

We will recall here certain facts about intersections of tropical curves.

The {\it multiplicity} $m(P)$ of the point $P$ of the transverse intersection of two lines in directions $(u_1,u_2),(v_1,v_2)\in\ZZ^2, gcd(u_1,u_2)=gcd(v_1,v_2)=1$ is  $|u_1v_2-u_2v_1|$.

Given two tropical curves $A,B\subset \TT^2$ we define their {\it stable intersection} as
follows. Let us choose a generic vector $v$. Then we consider the curves $T_{tv}A$
 where $t\in\RR, t\to 0$ and $T_{tv}$ is translation by
the vector $tv$. For a generic small positive $t$, the intersection $T_{tv}A\cap B$ is
transversal and consists of points $P_i^t, i =1,\dots,k$ with multiplicities $m(P_i^t)$.

\begin{prop}
\label{prop_stableint}
Suppose that a horizontal edge $E$ of a tropical curve $C$ contains a point $P$. Suppose that
on the dual subdivision of the Newton polygon for $C$ the vertical edge $d(E)$ is
dual to $E$. Let the endpoints of $E$ be $A_1,A_2$ and two faces
$d(A_1),d(A_2)$ adjacent to $d(E)$  have no other vertical edges.  
Let the sum of widths in the horizontal direction of the faces $d(A_1),d(A_2)$ be equal to $m$. Then the stable intersection of $E$ with a horizontal line through $E$ is $m$.
\end{prop}

\begin{proof} Refer to Example~\ref{ex_singularexample} and Figure~\ref{fig_stable}.
Let $L$ be a tropical line containing $E$, and assume that the vertex of $L$ does not coincide
with the endpoints of $E$. Making the modification along the line $L$ we see that 
the sum $S$ of vertical components
of edges going upward from $A_1,A_2$ equals the sum $m$ of the $y$-components of them. 

Then the sum of the vertical components of the downward edges equals \(S\) by
the balancing condition for tropical curves. The sum of the \(y\)-components of the edges incident to a vertex \(v\) is exactly the width, in the \((1,0)\)-direction, of the face \(d(v)\) dual to \(v\) in the Newton
 polygon.   
\end{proof}

\begin{definition}[cf. \cite{sturm}]
\label{def_stableintersection2}
For each connected component $X$ of $A\cap B$, we define {\it the
  local stable intersection of $A$ and $B$ along $X$} as
$A\cdot_XB=\sum_i m(P_i^t)$ for $t$ close to zero, where the sum
runs over $\{i|
\lim_{t\to 0} P_i^t\in X\}$. For a point $Q\in A$, we define $A\cdot_QB$ as
$A\cdot_XB$, where $X$ is the connected component of $Q$ in the
intersection $A\cap B$.
\end{definition}

\begin{prop}[\cite{brugalle} Proposition 3.11, see also \cite{MR2900439} Corollary 12.12]
\label{th_stableintersection}
For two algebraic curves \(C_1,C_2\subset(\KK^*)^2\) we consider a {\it compact} connected component $X$ of the intersection $\Trop(C_1)\cap \Trop(C_2)$. Then, $$\sum_{x\in C_1\cap C_2, \Val(x)\in X} m(x) = \Trop(C_1)\cdot_X\Trop(C_2)$$ where $m(x)$ is the multiplicity of the point $x$ in the intersection $C_1\cap C_2$.
\end{prop}

\begin{proof}
Consider the equation $F(x,y)=0$ of $C_2$. We construct the non-Archimedean amoeba $m_{C_2}C_1$ of $\{(x,y,F(x,y)\mid (x,y)\in C_1)\}$. Then $$\Trop(C_1)\cdot_X\Trop(C_2)$$ is the sum of the weights of the vertical legs of $m_{C_2}C_1$ under $X$. The latter is equal to $\sum_{x\in C_1\cap C_2, \Val(x)\in X} m(x)$.
\end{proof}

\begin{remark}
For non-compact connected components of the intersection we only have an inequality $\sum_{x\in C_1\cap C_2, \Val(x)\in X} m(x) \leq  \Trop(C_1)\cdot_X\Trop(C_2)$. It can be upgraded to equality by considering intersections of $C_1,C_2$ ``at infinity'', in the appropriate compactification of the torus, see \cite{shaw}.
\end{remark}

For further discussion about the multiplicity of singular points in the tropical world, see \cite{kalinin}.

\subsection{Tropical Weil reciprocity law and the tropical momentum map}
\label{weil}

The aim of this section is to establish another fact in tropical geometry, which is the tropical analogue of a classical fact in algebraic geometry.

Weil reciprocity law can be formulated as
\begin{thm}
\label{th_weil}
Let \(C\) be a smooth compact complex curve, and let \(f,g\) be nonzero
meromorphic functions on \(C\) whose divisors have disjoint supports. Then
\[
\prod_{x\in C} f(x)^{\ord_x(g)}
=
\prod_{x\in C} g(x)^{\ord_x(f)}.
\]
Here, if \(u\) is a local parameter at \(x\) and
\[
f=a u^n+\cdots,\qquad a\ne0,
\]
then \(\ord_x(f)=n\).
\end{thm}

The products in this theorem are finite because $\ord_gx, \ord_fx$ equal to zero everywhere
except finite number of points.

If \(f=a u^n+\cdots\) and \(g=b u^m+\cdots\) in a local parameter \(u\) at \(x\),
define the tame symbol
\[
[f,g]_x=(-1)^{nm}\frac{a^m}{b^n}.
\]
Then Weil reciprocity is equivalently
\[
\prod_{x\in C}[f,g]_x=1.
\]

\begin{example}
\label{ex_polynom}
If $C=\CC P^1$ and $f,g$ are polynomials \begin{equation}
f(x)=A\prod_{i=1}^n (x-a_i),g(x)=B\prod_{j=1}^m(x-b_j)
\end{equation} with $a_i\ne b_j$, then 
\[
\prod_{x\in C} g(x)^{\ord_x(f)}
=
\prod_{i=1}^n g(a_i)
=
B^n\prod_{i=1}^n\prod_{j=1}^m(a_i-b_j),
\]
and
\[
\prod_{x\in C} f(x)^{\ord_x(g)}
=
\prod_{j=1}^m f(b_j)
=
A^m\prod_{j=1}^m\prod_{i=1}^n(b_j-a_i).
\]
The remaining factor is exactly compensated by the tame symbol at infinity.

\end{example}

Khovanskii studied various generalizations of the Weil reciprocity law and reformulated them in terms
of logarithmic differentials \cite{weil1,Khovanskii1,weil_kh}. The final
formulation is for toric surfaces and resembles a tropical balancing condition, which is indeed the case.
The symbol $[f,g]_x$ is related with Hilbert character and link coefficient, and is generalized by Parshin residues. Mazin  \cite{mazin} treated them
 in geometric context of resolutions of singularities.

In order to study what happens after a modification we consider
a tropical version of Weil theorem. We need to define tropical meromorphic function and $\ord_fx$, see also \cite{mikhalkin2006tropical}.

\begin{definition}[\cite{mikh2}]
For a piecewise-linear function \(f\) with integer slopes on a tropical curve \(C\), define
\[
\ord_x(f)=\sum_{e\ni x} \operatorname{slope}_e(f),
\]
where the slopes are taken in the outgoing directions from \(x\). Points with
\(\ord_x(f)>0\) are zeros and points with \(\ord_x(f)<0\) are poles.
\end{definition}

\begin{example}
The function $f(x)=\max(0,2x)$ on $\TT P^1 =\{-\infty\}\cup\RR\cup\{+\infty\}$ has a zero of multiplicity $2$ at $0$, i.e. $\ord_f(0)=2$, and a pole of multiplicity $2$ at $+\infty$, i.e. $\ord_f(+\infty)=-2$.
\end{example}

\begin{thm}[A proof is in Section~\ref{sec_proofweil}]

\label{th_tropicalweil} Let \(C\) be a compact tropical curve and let \(f,g:C\to\RR\) be finite-valued
piecewise-linear functions with integer slopes. Then
\[
\sum_{x\in C} f(x)\ord_x(g)=\sum_{x\in C} g(x)\ord_x(f).
\] 
\end{thm}

Word-by-word repetition of the reasoning in Example~\ref{ex_polynom} proves this theorem in the case $C=\TT P^1$. Indeed, after adding a constant and, if necessary, a linear term, a one-variable tropical polynomial can be written in the form
\[
c+\sum_i \max(A_i,X),
\]
where the \(A_i\) are its tropical roots, counted with multiplicity.

For the general statement there are many proofs (and one can proceed by studying piece-wise
linear functions on a graph, see \cite{KalininMagin2025TropicalWeil}), we give here the shortest one (and also using tropical modifications), via so-called {\it tropical momentum}.

Suppose that $C$ is a planar tropical curve. We list all the edges $E_1,\dots,E_k$ of $C$, suppose that their directions are given by
primitive (i.e. non-multiple of another integer vector) integer vectors
$v_1,\dots,v_k$. Suppose that each edge $E_i$ has weight $m_i$ and if $E_i$ is infinite, then the direction of
$v_i$ is chosen to be ``to infinity'' (there are two
choices and for us the orientation of $v_i$ will be important). Let $A$ be a point on the plane. Let us choose a point 
$B_i\in E_i$ for each $i=1,2,\dots,k$. 
\begin{definition}[\cite{Yoshitomi:kq}]
Tropical momentum of an edge $E_i$ of $C$ with respect to the point $A$ is given by $\rho_A(E_i) =m_i\cdot \det(v_i,AB_i)$.
\end{definition}

\begin{definition}
For a point $A\in\RR^2$ define $\rho_A(C)$ as $\sum_{E}
\rho_A(E)$ where $E$ runs over all infinite edges of $C$.
\end{definition}

\begin{lemma}[\cite{Yoshitomi:kq}]
\label{lem_momentum}
If a tropical curve $C$ has only one vertex, then $\rho_A(C)=\sum_{i=1}^k
\rho_A(E_i) =0$ for any point $A$ on the plane.
\end{lemma}

\begin{proof} First of all, $\sum_{i=1}^k
\rho_A(E_i) $ does not depend on the point $A$, because if we translate $A$ by some vector $u$, then each summand in $\rho_A(C)$ will
change by $\det(v_i,u)\cdot w_i$ and the sum of changes is zero because of the balancing condition. Therefore, $\rho_A(C)=0$, because we can place $A$ in the vertex of
this curve.
\end{proof}

\begin{lemma}[Moment condition in \cite{Yoshitomi:kq}, also it appeared in \cite{Mikhalkin:2015kq} under the name Tropical Menelaus Theorem]
\label{lem_moment1}
For an arbitrary plane tropical curve $C\subset \RR^2$ and any point $A\in\RR^2$ the equality $\rho_A(C) =0$ holds.
\end{lemma}
\begin{proof}
Note that the total momentum
for a curve is the sum of momenta for all vertices, because a summand
corresponding to an edge between two vertices will appear two times
with different signs. So, this lemma follows from the previous one.  
\end{proof}

\begin{definition}
We consider a balanced tropical curve \(C\subset\RR^3\) with finitely many unbounded edges. Let $E_1,E_2,\dots,E_n$ be its infinite edges. We define the momentum of $C$ with respect to $A$ as 
\[
\rho_A(C)=\sum_{i=1}^n m_i\, v_i\times\overrightarrow{AB_i}.
\]
 where $\times$ stands for the vector product, $v_i$ is the primitive vector (in the direction ``to infinity'') of an edge $E_i$, $m_i$ is the weight of $E_i$, and $B_i$ is a point on $E_i$.
\end{definition}

\begin{prop}[Generalized Tropical Menelaus theorem]
\label{prop_moment1}
For a balanced tropical curve \(C\subset\RR^3\) and any point \(A\in\RR^3\), one has
\[
\rho_A(C)=0.
\]
\end{prop}

\begin{proof} We proceed as in the planar case. We show that $\rho_A(C)$ does not depend on $A$ because of the balancing condition. Indeed, if $C$ has only one vertex, then the claim is trivial. In general case we sum up the tropical momentum by all the edges, and the terms for internal edges appear two times with different signs, which concludes the proof.
\end{proof}

An application of this theorem can be found in Example~\ref{ex_infl}.

\subsection{Application of the tropical momentum to modifications}
\label{momentum}

\begin{example}
Consider the graph of a tropical polynomial $$f(X)=\max(A_0,A_1+X,\dots,A_n+nX).$$ Suppose that we know only $A_0$ and $A_n$. Definitely, the positions of the tropical roots of $f$ may vary, being dependent on the coefficients of $f$. Nevertheless, we can apply the tropical Menelaus theorem for the graph of $f$. We will calculate the momentum with respect to $(0,0)$. This graph has one infinite horizontal edge with momentum $A_0$ and one edge of direction $(1,n)$ with the momentum $-A_n$. Also, for each root $P_i\in\TT$ of $f$ we have an infinite vertical edge with the momentum $-P_i$. Application of the tropical moment theorem gives us $\sum P_i = A_0-A_n$, which is simply a tropical manifestation of Vieta's theorem --- the product of the roots $p_i$ of a polynomial $\sum_{i=1}^n a_ix^i$ is $a_0/a_n$.
\end{example}

\begin{example}
\label{ex_weil}
Let $C$ be a planar tropical curve, such that all its infinite edges are horizontal or vertical. Consider first and second coordinates $X,Y$ on $C$ as two tropical functions. Denote these functions $f=X,g=Y$. Then, Theorem~\ref{th_tropicalweil} says that $\rho_{(0,0)}C = 0$, because a tropical root of $f$ is represented by a horizontal leg of $C$, and the value of $g$ at this root is exactly the $Y$-coordinate of this leg. 
\end{example}

On Figure~\ref{fig_stable}, \ref{fig_inflection}, {\it a priori} we know only the sum of the directions of the edges with
endpoints on the modified curve. We know that there is no
horizontal infinite edges (in these examples). In general, it is
possible, if the intersection of our two tropical curves is non-compact. Therefore by Weil theorem (or tropical Menelaus Theorem, it is
the same) we know the sum of $X$-coordinates of the vertical infinite
edges. Thus the sum of the weights for red vertical edges equals the sum of the vertical
components of the black edges in the Figure~\ref{moment}.

\begin{lemma}
If we modify along a horizontal line, then the total vertical slope of the infinite vertical edges lying over this line equals the total horizontal width of the region in the Newton subdivision dual to the corresponding connected component of the intersection of the line with the curve.
\end{lemma}
\begin{proof}
The argument is the same as in the proof of Proposition~\ref{prop_stableint}.
\end{proof}

\begin{lemma}
\label{lem_positionsingular}
If the stable intersection of $\Trop(C)$ with a horizontal line $L$ is equal to $m$, $\Trop(C)\cap L$ is compact, and there exists a point $q\in C$ with $\mu_q(C)\geq m$ and 
$\Val(q)\in\Trop(C)\cap L$, then we can uniquely recover the position of $\Val(q)$. 
\end{lemma}
\begin{proof}
Indeed, consider a lift $l$ of $L$ which passes through $q$. If we make the modification along $l$, we obtain a leg of $m_L(\Trop(C))$ under $\Trop(C)\cap L$ of the weight at least $m$. Since the stable intersection $\Trop(C)\cap L$ is equal to $m$, there is only one leg under $\Trop(C)\cap L$. Therefore, the tropical momentum theorem gives us the unique position of this leg (of course, it is evident via balancing --- we know all the infinite edges of a tropical curve except one, therefore the coordinates of this last edge can be found via the balancing condition).
\end{proof}

\begin{prop}[see \cite{brugalle}, Proposition 4.5]
\label{prop_brug}
For each compact connected component $C$ of $C_1\cap C_2$ the sum of $X$ coordinates (and the sum of $Y$-coordinates) of the valuations of the intersection points of $C_1,C_2$ with valuations in $C$ can be computed from the local behavior of \(C_1\) and \(C_2\) near \(C\).
\end{prop}
Indeed, the tropical Menelaus theorem gives the sum of the momenta of all legs of \(m_{C_2}C_1\) whose \(Z\)-coordinate tends to \(-\infty\).

\subsection{Proof of the tropical Weil theorem}
\label{sec_proofweil}
We now prove the tropical Weil theorem. Given two tropical meromorphic functions \(f,g\)
on a tropical curve \(C\), we would like to define the map
\[
C\to\TT^2,\qquad x\mapsto (f(x),g(x)),
\]
and then apply the tropical Menelaus theorem; cf. Example~\ref{ex_weil}. Here we have to use tropical modification, because {\it a priori}, the image of tropical curve under the map $(f,g):C\to\TT^2$ with $f,g$ tropical meromorphic functions, is {\bf not} a plane tropical curve: balancing condition is not satisfied near zeroes and poles of $f$ and $g$, we need to add legs there. Formally, we have to consider a modification $C'$ of $C$, and then extend $f,g$ on it. Then, if the roots and poles of $f,g$ will be only at $1$-valent vertices, then the image of the map $C'\to \TT^2$ will be a planar tropical curve.

\begin{definition}
We call a triple $(C,f,g)$ of a tropical curve $C$ and two meromorphic function $f,g:C\to \TT P^1$ on it {\it admissible} if all the zeroes and poles of $f,g$ are located at different one-valent vertices of $C$. 
\end{definition}

\begin{lemma}
Given a triple $(C,f,g)$ of a tropical curve $C$ and two meromorphic functions \(f,g:C\to \TT P^1\) on it, we can always extend the functions \(f,g\) on the modification $D=m_{div(g)}m_{div(f)}C$ of $C$, such that the obtained triple $(D,f',g')$ is admissible and  \begin{equation}
\sum\limits_{x\in C} f(x)\operatorname{ord}_x(g) -
\sum\limits_{x\in C} g(x)\operatorname{ord}_x(f)
=
\sum\limits_{x\in D} f'(x)\operatorname{ord}_x(g') -
\sum\limits_{x\in D} g'(x)\operatorname{ord}_x(f').
\end{equation}
\end{lemma}

\begin{proof}
We perform tropical modifications of $C$ in order to have all zeros and poles of $f,g$ at the vertices of valency one. Namely, for a point $p$ such that $p$ is in the corner locus of $f$ we add to $C$ an infinite edge $l$ emanating from $p$. We define $f$ on $l$ as the linear function with integer slope such that the sum of slopes of  $f$ over the edges from $p$ is zero, i.e. $f(x)=f(p)-\ord_fp\cdot x$ where $x$ is the coordinate on $l$ such that $x=0$ at $p$ and then $x$ grows. We define $g$ on this edge as the constant $g(p)$. We perform this operation for all roots and poles of $f$. Then, we do the same procedure along the divisor of $g$. If \(p\) belongs to both divisors, we add one infinite leg \(l\) and extend \(f\) and \(g\)
linearly on \(l\) so that, under the map \((f,g)\), this leg has primitive direction
proportional to \((\ord_p(f),\ord_p(g))\). With this choice the balancing defects of
both \(f\) and \(g\) at \(p\) are moved to the new one-valent end.
\end{proof}

\begin{proof}[Proof of the tropical Weil theorem]
By the lemma above we may suppose that the triple $(C,f,g)$ is admissible.
Now $f,g$ define a map $C\to \TT^2$ and the image is a tropical curve $D=\{(f(x),g(x))|x\in C\}$: indeed, at every vertex of the image the balancing condition is satisfied; all one-valent vertices go to infinity by one of the coordinates. Now it is easy to verify that $g(x)\cdot \ord_f(x)$ is a term in the definition of the momentum of $D$ with respect to $(0,0)$: if $\ord_f(x)\ne 0$, then $D$ has a horizontal infinite edge, and its $Y$-coordinate is $g(x)$. Finally,
\begin{equation}
\label{eq_momentum}
\sum\limits_{x\in C} f(x)\operatorname{ord}_x(g) -
\sum\limits_{x\in C} g(x)\operatorname{ord}_x(f)
= \rho((0,0))=0.\qedhere
\end{equation}
\end{proof}

\begin{remark}
If $f,g$ come as limits of complex functions $f_i,g_i$, having $\ord_{f_i}(p_i)=k,\ord_{g_i}(p_i)=m, \lim p_i=p$, then the tropical limit of the family $\{(f_i(x),g_i(x))|x\in C_i\}$ will not have vertical (with multiplicity $k$) and horizontal (with multiplicity $m$) leg from a common divisor point $p$ of $f$ and $g$, but will have one leg of direction $(k,m)$. Nevertheless, because of the tropical Menelaus theorem or the balancing condition, it has no influence on Eq.~\eqref{eq_momentum}.
\end{remark}

\subsection{Stable intersections, realizable intersections}
\label{sec_realizability}

One may ask if the only obstruction for a
modification is the generalized tropical Menelaus theorem. As we will see in this section, not at all.

Let us start with a variety $M'\subset \KK^n$ and a hypersurface
$N'\subset \KK^n$ and their non-Archimedean
amoebas $M, N\subset \TT^n$. We suppose that the intersection of $M$ with a tropical hypersurface $N$ is
not transverse. We ask: what can the non-Archimedean amoeba of the intersection \(M'\cap N'\) look like?

First of all, as a divisor on $M$ (or $N$) it should be rationally equivalent to the divisor of the stable
intersection of $M$ and $N$, as it has been shown for the case of curves in
\cite{Morrison:2014mz}. In the general case it follows from the
results of this section.

It is easy to find additional necessary conditions. Let us restrict the equation \(F\) of \(N'\) to \(M'\) and tropicalize the resulting data. We obtain a tropical function \(f=\Trop(F)\) whose behavior in a neighborhood of \(N\cap M\) is fixed, whereas its behavior on all of \(M\) is not.

\subsection{A realizability obstruction for tropical intersections}
The next result is one of the main points of the paper.  It gives an obstruction
to realizability of tropical intersection divisors.  Stable intersection gives
the divisor obtained when no cancellation occurs in the lifted equations.
However, if the algebraic lifts have cancellations, the valuation of the
intersection may move inside the stable-intersection component.  The theorem
below says that this movement is one-sided: the resulting divisor must be
subordinate to the stable intersection.

\begin{definition}
\label{def_frozen}
Let \(M\) be an abstract tropical variety, and let \(\iota:M\to\TT^n\) be an embedding realizing \(M\) as a tropical subvariety of \(\TT^n\). Let \(f\) be a tropical function on \(\TT^n\). We define the pullback \(\iota^*(f)\) on \(M\) by \(\iota^*(f)=f\circ\iota\). We say that \(\iota^*(f)\) is frozen at \(p\in M\) if, in a neighborhood of
\(\iota(p)\) in \(\TT^n\), the tropical polynomial \(f=\Trop(F)\) is represented
by a single monomial. Equivalently, \(\iota(p)\) is not in the corner locus of
\(f\).
\end{definition}

Note that, in general, the slopes of \(f\) along \(\iota(M)\) do not coincide with the slopes of \(\iota^*(f)\) on \(M\); see Example~\ref{ex_stability}.
From now on we consider tropical functions which have frozen points, the motivation is explained in the following definition.

\begin{definition}
\label{def_subordinate}
A principal divisor $P$ on an abstract tropical variety $M$ is called {\it subordinate to} a principal
divisor $Q$ (we write $P\prec Q$), which is defined by a tropical meromorphic function $f$ with frozen points, if $P$ can
be defined by a tropical meromorphic function $h$, which satisfies $h\leq f$ everywhere and $h=f$ at the points where $f$ is frozen. 
\end{definition}
Thus a subordinate divisor is obtained by replacing the tropical function
\(f\) by a smaller tropical meromorphic function \(h\), while keeping \(h=f\)
on the frozen locus.  In the intersection problem, \(f\) is the pullback of the
tropical equation defining the stable intersection, and \(h\) records the
possible cancellations in a lift.

\begin{remark}
On a compact connected tropical curve, after fixing the frozen locus, the condition is invariant under adding the same constant to both defining functions. In higher-dimensional or noncompact settings, one should regard the frozen locus and the chosen defining function \(f\), up to the appropriate additive ambiguity, as part of the data.
\end{remark}

\begin{example}
\label{ex_stability}
Refer to Example~\ref{ex_singularexample}. Let us start from the tropical curve $M$ given by 
\begin{align*}
\max(3+X+3Y,&3+X+2Y,3+X+Y,3+X,2+2X+2Y,2+2X+Y,\\
&2+2X,1+Y,1,3X-2)
\end{align*} 
and a horizontal line $N$ given by $\max(Y,0)$. We want to understand the valuations of possible intersections of $M'\cap N'$ where $\Trop(M')=M,\Trop(N')=N$.

We can choose the equation for $M'$ in the form 
\begin{equation*}
F(x,y)=(t^{-1}+\alpha_0+t^{-1}y) + x(t^{-3}+\alpha_1+t^{-3}y^3)+x^2(t^{-2}+\alpha_2+t^{-2}y)+ x^3(t^{2}+\alpha_3),
\end{equation*} where $\val(\alpha_0)<1,\val(\alpha_1)<3, \val(\alpha_2)<2,\val(\alpha_3)<-2$. It is clear, that for any $A\leq 1, B\leq 3, C\leq 2$ by choosing $y$ of the form $1+\alpha, \val(\alpha)<0$ and then with careful choice for $\alpha_1,\alpha_2,\alpha_3$ we can obtain  (see Figure~\ref{fig_39})
\begin{equation*}
f(X)=\Val(F(x, 1+\alpha))=\max(A,B+X,C+2X,-2+3X).
\end{equation*}

In this example the set $X\geq 4$ on $N$ is frozen for $\Trop(F)$, that is why we have a choice for the constant term $A$. If the intersection is a compact set (as in Example~\ref{ex_big}), then the constant term is also fixed.
Note that for the stable intersection our tropical function is $\Trop(F)(X,0)=\max(1,3+X,2+2X,-2+3X)$ and $f(X)\leq\Trop(F)(X,0)$ at every point.

\end{example}

\begin{figure}
\begin{tikzpicture}[scale=0.7]
\draw[red,very thick] (-2,0)--(-3.5,-0.5)--(-4,-0.5);
\draw[red,very thick] (-3.5,-0.5)--(-3.5,-3);
\draw[red,very thick] (4,0)--(3.5,-0.5)--(3.5,-3);
\draw[red,very thick] (3.5,-0.5)--(0.5,-0.5)--(1,0);
\draw[red,very thick](0.5,-0.5)--(-0.66,-1.33)--(-2,0);
\draw[red,very thick](-0.66,-1.33)--(-0.66,-3);
\draw[blue, thick] (-4,0.05)--(6,0.05);
\draw[very thick] (-2,0)--(4,0)--(6,1);
\draw (0,0) node[above] {$2$};
\draw (2,0) node[above] {$1$};

\draw[very thick] (-4,1)--(-2,0);
\draw[very thick] (2,2)--(1,0);

\begin{scope}[scale=0.6, xshift=430]
\draw[dashed](-3,0)--(4,0);
\draw[dashed](0,-3)--(0,5);

\draw (-3,-5/6)--(-2.5,-5/6)--(-2/3,5/12-2/6)--(3.5,8.5/2)--(5,13/2);
\draw (-2.5,-5/6)--(-2.5,-3);
\draw (-2/3,5/12-2/6)--(-2/3,-3);
\draw (3.5,8.5/2)--(3.5,-3);

\end{scope}

\end{tikzpicture}
\caption{Refer to Example~\ref{ex_stability}. On the left figure we see the vertical part of the modification of the curve given by $F(x,y)=(-t^{-1}+t^{5/3}+t^{-1}y)+x(t^{-3}y -(t^{-3}+t^{-5/6})+x^2(t^{-2}y-t^{-2}+t^{-3/2})+x^3t^2$ along the line $y=1$. On the right figure, we see the tropicalization of the restriction of $F$ on $y=1$, i.e., the function  $\max(3X-2,2X+1.5,X+5/6,-5/3)$. }
\label{fig_39}
\end{figure}
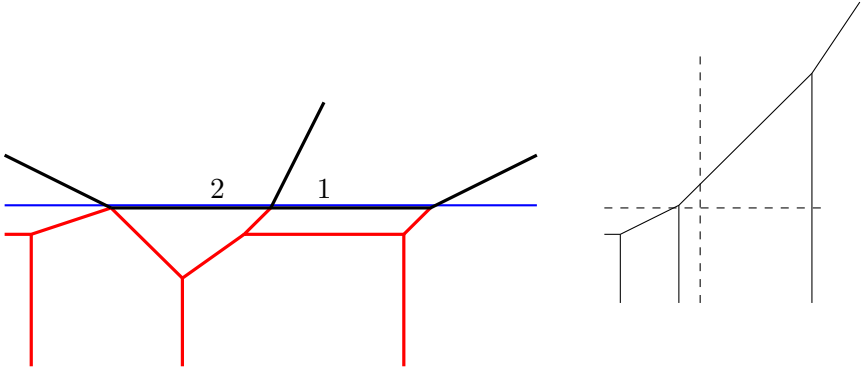

Now we prove the following theorem, whose proof consists only of a reformulation in the language of tropical modifications and staring to the pictures; see Remark~\ref{rem_restriction} as an illustration.
 
Fix an abstract tropical variety $M$, its tropical embedding $\iota:M\to\TT^n$, and a tropical hypersurface $N\subset\TT^n$, given by a tropical polynomial $f$. As we know, the pullback of the divisor of the stable intersection of $\iota(M)$ with $N$ is given by $\iota^*(f)$. Note that we supply the function $\iota^*(f)$ with frozen points, according to Definition~\ref{def_frozen}.

\begin{thm}[Subordination obstruction for lifted intersections]
\label{th_subordinate}
Let \(M'\subset(\KK^*)^n\) be a pure-dimensional algebraic variety, and let
\(N'\subset(\KK^*)^n\) be the hypersurface defined by \(F=0\). Assume that
\(F|_{M'}\) is not identically zero on any irreducible component of \(M'\), and
that \(N'\cap M'\) is a proper Cartier divisor on \(M'\).

Let \(M=\Trop(M')\), let \(N=\Trop(N')\), and suppose that \(M\) is identified
with an abstract tropical variety via an embedding
\[
\iota:M\to\TT^n.
\]
Let \(f=\Trop(F)\), and equip \(\iota^*(f)\) with the frozen locus from
Definition~\ref{def_frozen}. Then the pullback to \(M\) of the divisor
\[
\Val(N'\cap M')
\]
is subordinate to the stable-intersection divisor
\[
\operatorname{div}(\iota^*f).
\]
In other words, every realizable tropical intersection divisor lies below the
stable intersection in the sense of Definition~\ref{def_subordinate}.
\end{thm}
Consequently, if a divisor rationally equivalent to the stable intersection is
not subordinate to it, then it cannot be obtained as the valuation of an
intersection of algebraic lifts with the prescribed tropicalizations.

\begin{proof} Recall that $f=\Trop(F), f:\TT^n\to\TT$. Let us make the modification of $\TT^n$ along $N$. Look at the image $m_f(M)$ of $M$ under this map. Clearly, the valuation of the set $\{(x,F(x))|x\in M'\}$ belongs to $m_f(M)$. Define
\[
h(X)=\max\{\val(F(x))\mid x\in M',\ \Val(x)=X\}.
\]
Then the graph of \(h\) is contained in the tropical modification \(m_f(M)\), and
\(h(X)\le \iota^*f(X)\) for all \(X\in M\).

At points \(X\in M\) where a unique monomial of \(F\) realizes the maximum
defining \(f\), no cancellation is possible, hence \(h(X)=\iota^*f(X)\).
These are precisely the frozen points in the sense used above.

  Therefore the pullback of $\iota^*(\Trop(F|_{M'}))$ is at most $\iota^*(f)$ everywhere, and $\iota^*(\Trop(F|_{M'}))=\iota^*(f)$ at the points where $\iota^*(f)$ is frozen. Therefore the divisor of \(h\) on \(M\) is subordinate to the divisor of
\(\iota^*(f)\), which is the pullback of the stable intersection divisor.
\end{proof}
The graph of \(h\) can lie below the graph of \(\operatorname{Trop}(F)|_M\),
because cancellations may occur after substituting points of \(M'\) into \(F\), which is invisible when we consider $\Trop(F)$ as a function on $\TT^n$. Recall that if the image of the valuation map $\val$ is $\TT$, then we know that $\Trop(F)(X)$ is the maximum of $\val(F(x))$ with $\Val(x)=X$. On the other hand, $\Trop(F_{M'})(X)$ for $X\in M$ is the maximum of $\val(F(x))$ with $\Val(x)=X$ and $x\in M'$. Clearly, the latter maximum is at most the former maximum.

\begin{corollary}[Non-subordination implies non-realizability]
\label{cor_nonsubordinate}
Let \(D\) be a divisor on \(M\) rationally equivalent to the stable-intersection
divisor \(\operatorname{div}(\iota^*f)\). If \(D\) is not subordinate to
\(\operatorname{div}(\iota^*f)\), then \(D\) cannot be the valuation of
\(N'\cap M'\) for algebraic lifts \(M'\) and \(N'\) with the prescribed
tropicalizations.
\end{corollary}

\begin{example}
The following example illustrates the obstruction in Theorem~\ref{th_subordinate}.
Refer to Figure~\ref{fig_subordinate}. We have the stable intersection $A+B+C+D$ of the curves given by $\max(0,Y)$ and $\max(0,X,2X-1,3X-3,4X-6,X+Y,2X+Y-1,3X+Y-3)$. The divisor \(A+B'+C'+D\) is rationally equivalent to the stable-intersection
divisor \(A+B+C+D\).  However, the rational function moving
\(A+B+C+D\) to \(A+B'+C'+D\) forces the corresponding tropical function to exceed
\(\iota^*(f)=\max(0,X,2X-1,3X-3,4X-6)\) on the frozen region \(X\ge4\).  Hence this divisor is not
subordinate to the stable intersection, and Theorem~\ref{th_subordinate}
implies that it is not realizable as the valuation of an intersection of lifts. This example shows that our conditions are stronger than in \cite{Morrison:2014mz}.
\end{example}

Note that to realize a divisor subordinate to the stable intersection as a tropicalization of intersection one needs to tune the coefficients of $F$ to have all necessary cancellations. For the case of a planar curve this can be done along trees, \cite{sukenaga2023tropical}. If a curve contains cycles then new conditions appear.

\begin{figure}[htbp]

\begin{center}
\begin{subfigure}[b]{0.45\textwidth}
\begin{tikzpicture}
\draw[blue, thick] (-1,0)--(4,0);
\draw[thick] (-1,1)--(0,0)--(0,-2);
\draw[thick] (1,1)--(1,-2);
\draw[thick] (2,1)--(2,-2);
\draw[thick] (0,0)--(3,0);
\draw[thick] (4,1)--(3,0)--(3,-2);
\draw (0,0) node[above]{$A$};
\draw (1,0) node[above right]{$B$};
\draw (2,0) node[above left]{$C$};
\draw (3,0) node[above]{$D$};
\draw (0,0) node {$\bullet$};
\draw (1,0) node {$\bullet$};
\draw (2,0) node {$\bullet$};
\draw (3,0) node {$\bullet$};

\draw[red] (0.7,0) node {$\bullet$};
\draw[red] (2.3,0) node {$\bullet$};
\draw[red] (0.7,0) node[above]{$B'$};
\draw[red] (2.3,0) node[above]{$C'$};
\end{tikzpicture}
\end{subfigure}
\begin{subfigure}[b]{0.45\textwidth}
\begin{tikzpicture}[scale=0.7]

\draw[blue, thick] (-1,0)--(4,0);
\draw[thick] (-1,1)--(0,1)--(1,2)--(2,4)--(3,7)--(4,11);
\draw[red, dashed] (0.7,0)--(0.7,1.7)--(2.3,4.9)--(2.3,0);

\draw[thick] (0,1)--(0,0);
\draw[thick] (1,2)--(1,0);
\draw[thick] (2,4)--(2,0);
\draw[thick] (3,7)--(3,0);

\end{tikzpicture}
\end{subfigure}
\end{center}
\caption{The divisor $A+B'+C'+D$ is not realizable as the valuation of intersection of the lifts of the curves defined by $\max(0,Y)$ and $\max(0,X,2X-1,3X-3,4X-6,X+Y,2X+Y-1,3X+Y-3)$. On the right we see the function which carries this rational equivalence out, it is bigger than the function for the stable intersection and so violates the subordination condition of  Theorem~\ref{th_subordinate}.}
\label{fig_subordinate}
\end{figure}
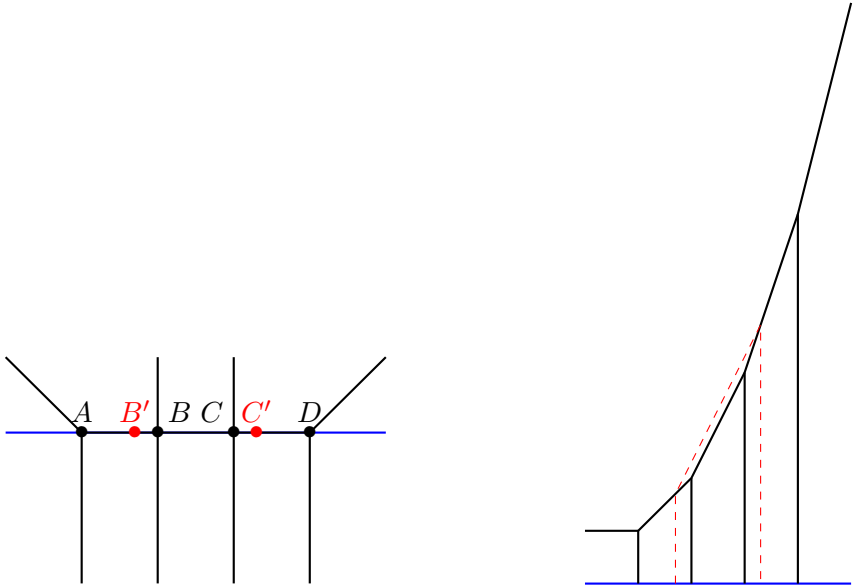

\subsection{Interpretation with chips}
In the case of curves we can represent a divisor on a curve as a collection of chips. In the last subsection we proved Theorem~\ref{th_subordinate} which says that any realizable intersection is subordinate to to the stable intersection. So, one might ask for a method to produce {\bf all} the subordinate divisors to a given divisor (though, it is possible that not all of them are realizable as the valuation of an intersection).

Let us start with the stable intersection of two tropical curves, this intersection is a divisor (collection of chips) on the first curve. Then we allow the following movement: pushing continuously together two neighbor chips on an edge, with equal speed. We do not allow the opposite operation --- when we slide continuously two points apart from each other (so, the operation in Figure~\ref{fig_subordinate} does not provide a subordinate to $A+B+C+D$ divisor).  

This corresponds to the following: we look at the modification of the first curve along the second curve, given by a tropical polynomial $\Trop(F)$. By decreasing the coefficients of the monomials in $\iota^*(f)$ on $C$, one by one, we can obtain any function less than $\iota^*(f)$.

This reasoning can be applied to the intersection of any two tropical
varieties if one of them is a complete intersection. We restrict the equations of
the second variety on the first, which gives us a stable intersection,
then we have a situation similar to Definition \ref{def_subordinate}, and, as above, by decreasing the coefficients of pullbacks of tropical polynomials we can obtain all the subordinate to the stable intersection divisors.

\begin{example}
Consider the function $\max(0,X-1,2X-3)$. This function defines the divisor on $\TT^1$ with two chips, one at $X=1$ and the second at $X=2$. When we decrease the coefficient in the monomial $X-1$, these chips are moving closer. For example, the function $\max(0,X-1.3,2X-3)$ defines the divisor with chips at the points with the coordinates $1.3$ and $1.7$.
\end{example}

\begin{remark}
\label{rem_infinitechips}
Note that if the stable intersection is not compact, then we need to add a chip at infinity (or to treat infinity as a point with one chip). Now let $A,B$ be two chips, $A$ is at infinity and $B$ is on the leg of $V$ going to $A$. Then, the operation ``pushing together $A,B$''  moves only $B$ towards infinity (and $A$ remains unchanged at infinity). This corresponds to decreasing the constant term in 
Example~\ref{ex_stability}.
\end{remark}

\example{Big order tangency with only two degrees of freedom}. (\cite{brugalle}, Lemma 3.15). We consider a line $y-\alpha x-\beta=0, \val(\alpha)=0,\val(\beta)=0$ and a curve $a_0+a_1y+a_2xy^l=0$ with $\val(a_0)=0,\val(a_1)=0,\val(a_2)=0$.

Clearly, we have non-transversal intersection, we can perform substitution $y=\alpha x+\beta$, that gives $a_0+a_1(\alpha x+\beta)+a_2x(\alpha x+\beta)^l = (a_0+a_1\beta)+x(a_1\alpha+a_2\beta^l)+\sum_{i=2}^{l+1}  a_2 \beta^{l+1-i}\alpha^{i-1} x^i$. The contraction may only appear at two coefficients: the coefficient before $x$ and the constant term. So we have only two degrees of freedom. Let us present the intersection points as chips. By changing the coefficients $\alpha,\beta,a_i$ we change the intersection, so we can look at how the chips move. So, when $\val(a_0+a_1\beta)<\val(a_0)$, this corresponds to the movement in Remark~\ref{rem_infinitechips}, one chip moves towards infinity while the others do not move at all. Also we can push two chips together by decreasing the valuation of $a_1\alpha+a_2\beta^l$. Note that $l-2$ chips at the point $(0,0)$ are unmovable.

Here we have only two degrees of freedom because we have only two degrees of freedom in the equation $a_0+a_1y+a_2xy^l=0$. 

\begin{problem}
\label{pr_intersection}
Motivated by the above example, we give the following suggestions, which seem to be reasonable in the question of the realizability of intersections. Suppose that we have a tropical line and a tropical curve defined by a tropical polynomial $f$. While defining $\iota^*(f)$ we keep track of all the monomials $m_i$ of $f$ and then in Definition~\ref{def_subordinate} we allow $g$ to contain only monomials of the type $\iota^*(m_i)$.
I.e. if $f=\max (a_{ij}+i X+ iY)$, then we only allow $g$ of the type $\max (c_{ij}+\iota^*(x^iy^j))$ with $c_{ij}\leq a_{ij}$ which coincides with $f$ on the frozen set of $f$. We explain why we concentrate on the case when one of the curves is a line. Normally, we can perturb the coefficients of the equations of both curves. If one of the curves is a line, we can always suppose that its equation is fixed.
For the general case, one should expect that apart from $\iota^*(f)$ on $M$ we can find another thin structure, which is responsible for the deformation of the equation of $M$ being immersed to $\TT^n$, something like ``a pull-back of the normal bundle'', coming from the map $\iota$. As we mentioned before, this program can be carried for the case of planar curves \cite{sukenaga2023tropical}.
\end{problem}

\example{Difference between a leg of big weight and a root.}
\label{ex_tangent} Take the curve $C$ given by $F=0$ where 
$F(x,y)=1+(t^{-1}+t)x+(2t^{-1}+t^2+t^4)x^2+(t^3+2t^4)x^3+t^{-1}xy+2t^{-1}x^2y$
and intersect it with the line given by $t^5x+y+1=0$.

Performing the tropical modification along the line we see that the resulting curve has a leg of weight three going to $-\infty$. But it is not a root of multiplicity three! If we substitute $y=-1-t^5x$ to the equation, we will see that the obtained polynomial $1+tx+t^2x^2+t^3x^3$ has three roots with the valuation $1$, but they do not coincide.
But if we consider the curve $C'$ given by the equation $F=0, F(x,y)=1+(t^{-1}+3t)x+(2t^{-1}+3t^2+t^4)x^2+(t^3+2t^4)x^3+t^{-1}xy+2t^{-1}x^2y$, we see that $\Trop(C)=\Trop(C')$ and $C'$ has a tangency of order three with the line.

The same example can be constructed for a similar Newton polygon
$$\mathrm{ConvHull}\{(0,0)-(1,1)-(n,1)-(n+1,0)\},$$ where we also can obtain the tangency of the order $n+1$.

\begin{problem}
\label{pr_tangency}
Suppose that the intersection of a tropical line with a tropical curve is a segment. Is it always possible to make a modification in order to have a leg of the weight equal to the local stable intersection (Definition~\ref{def_stableintersection2})? If yes, is it always possible to find the coefficients for the equations in order to have a tangency of the order equal to the stable intersection? Also, we can ask this question for any two curves with non-transverse intersections.
\end{problem}

Due to combinatorial restrictions in tropical terms, sometimes we can see that it is impossible to have a singular point with high multiplicity on a curve. Note that even in this case, we can have a leg of big multiplicity after the modification; see Example~\ref{ex_tangent}.

%

\subsection{Digression: a generalization of the tropical momentum}
\label{sec_digression}
A natural generalization of the vector product (or cross product) in $\RR^3$ 
\begin{equation}
(x_1,y_1,z_1)\times (x_2,y_2,z_2)
=
(y_1z_2-z_1y_2,\ z_1x_2-x_1z_2,\ x_1y_2-y_1x_2)
\end{equation} is the following. Given $k$ vectors $v_1,v_2,\dots,v_k\in\RR^n, k\leq n$ we consider the vector consisting of all \(k\times k\) minors of the \(k\times n\) matrix whose rows are \(v_1,\dots,v_k\). We call this vector of minors {\it generalized cross product of $v_1,v_2,\dots,v_k$}.

Consider a tropical variety $V^k\in\TT^n, k<n$. Let us choose a basis in each face of $V$ of codimension one and zero, i.e. for a face $F$ we choose a basis in the lattice associated with the integer affine structure of this face. For each face $G$ of codimension one in $V$ and the faces $F_1,F_2,\dots,F_l$ of codimension zero, containing $G$, we choose vectors $v_G(F_i)$ which participate in the balancing condition along $G$. Now we can define the sign $s_G(F)\in\{+1,-1\}$ to be $+1$ if the basis in $G$ with added vector $v_G(F)$ at the last place gives the same orientation in $F$ as the previously chosen basis in $F$, and $-1$ otherwise.

\begin{definition}
Let $\mathfrak G(V)$ be the abelian group generated formally by all the faces $G$ of $V$ of codimension one. Now we will describe relations in it. For a face $F\subset V$ of maximal dimension define $m(F)\in \mathfrak G(V)$ to be the sum $\sum_{G\subset F}s_G(F)\cdot G$. For each bounded face $F\subset V$ of maximal dimension we add the relation $m(F)$  to $\mathfrak G(V)$.
\end{definition}

\begin{example}
Compare Definition~\ref{def_generalizedmoment} with the proof of Lemma~\ref{lem_moment1}. For the case of planar tropical curve $C$ the group $\mathfrak G(V)$ is generated by all the vertices of $C$. Then, each internal edge of $C$ gives the relation that its ends are equal. Therefore, in that case, the group $\mathfrak G(V)$ is $\ZZ$ with generator $\mathfrak 1$, and for each unbounded $F$, we have $m(F)=\mathfrak 1$.
\end{example}

\begin{definition}
\label{def_generalizedmoment}
Let $A$ be a point in $\TT^n$. Pick a face $F$ of $V$ of codimension zero and let $B$ be a point in $F$. Then, define $r_A(F)$ as the generalized cross product of the vector $AB$ and the vectors in the basis in $F$. Note that $r_A(F)$ does not depend on $B$. Finally, define 
\begin{equation}
\rho_A(F)=r_A(F)\otimes_{\QQ}m(F)\in \RR^{\binom{n}{k+1}}\otimes_{\QQ}\mathfrak G(V).
\end{equation}
\end{definition}

\begin{prop}
For any point $A\in\TT^n$ we have $\sum_F \rho_A(F)=0$, where $F$ runs over all the unbounded faces of $V$ of the maximal dimension.
\end{prop}
\begin{proof}
The structure of the proof is the same as in Lemma~\ref{lem_momentum}. Let us only show that $\sum \rho_A(F)$ does not depend on the point $A$. Indeed, for each face $G\subset V$ of the codimension {\bf one} we consider the terms in  $\sum \rho_A(F)$ which contain $G$. It is easy to see, that thanks to the balancing condition along $G$ and our choice of signs, the sum of these terms is zero. 
\end{proof}

\begin{problem}
It seems that in a general situation, if $V$ is a tropical curve, then, again, $\mathfrak G(V)$ is $\ZZ$. On the other hand, it seems that if the dimension of $V$ is at least two, then $\mathfrak G(V)$ is freely generated by the unbounded faces of $V$ of codimension 1. Also, it would be nice to state an analog of the tropical Weil theorem in this new context and find its classical algebraic counterpart. 
\end{problem}

\section{Applications of a tropical modification as a method}

The term ``modification'' is used here in two related senses: first, as a well-defined {\it operation} on tropical varieties, as in \cite{mikh2}; and second, as a {\it method} for revealing the behavior of other varieties in an infinitesimal neighborhood of \(N\). Namely,
performing the modification of $M$ along $N\subset M$, we will see how $M$
changes, but the objects of codimension one in $M$ may behave
differently, depending on their behavior near $N$.  We will clarify this
distinction with examples.

\subsection{Inflection points}
\label{sec_inflection}
For a smooth plane curve, an inflection point is a smooth point at which the
tangent line has contact order at least \(3\). 
For a nonsingular real plane curve of degree \(d\), Klein's bound says that the
number of real inflection points is at most \(d(d-2)\), and this bound is sharp. The question attacked in
\cite{brugalle} is {\it which topological types of planar real algebraic
curves admit the maximal number of real inflection points?} Using Viro's patchworking method, a classical construction of algebraic curves, the authors construct examples. For this purpose, they study the possible local pictures of tropicalizations of inflection points.  The property to be verified is tangency, but intersection of
tropical curve with a tangent line at some point in most cases is not
transversal and it is not visible what is the actual order of
tangency. To see that, the authors do tropical modifications.

\subsection{The category of tropical curves}
\label{category}

For the treatment of this question with tropical harmonic maps see \cite{MR3375652, MR3320845}.
G. Mikhalkin (lectures, 2011) defines the morphisms in
the category of tropical curves as all the maps, satisfying the balancing and
Riemann-Hurwitz conditions (see, for example \cite{MR2866125}) and
subject to the {\it modifiability} condition:

\begin{definition}
A morphism $f:A\to B$ of tropical curves $A,B$ is said to be {\it modifiable} if for any modification $B'$ of $B$ there exists a
modification $A'$ of $A$ and a lift $f'$ of $f$ which makes the
obtained diagram commutative. 
\end{definition}

\begin{remark}
The modifiability condition ensures that a morphism came as a
degeneration of maps between complex curves (see Section~\ref{sec_hyperbolic}).
\end{remark}
\begin{proof}[Sketch of a proof] After a number of modifications we may have the map
$f'$ contracting no cycles. Then we construct a family of
complex curves $B_i$ such that $\lim B_i= B'$ in the hyperbolic sense (see
section \ref{sec_hyperbolic}). Finally, since $f'$ should come as a tropicalization of a covering, the complex
curves $A_i$ with $\lim A_i=A'$ are constructed as coverings $f_i:A_i\to B_i$ over $B_i$
where the combinatorics (ramification profiles, local degrees at points of tori contracting to tropical edges) of $f_i$ is prescribed by
$f'$. Balancing and Riemann-Hurwitz conditions follow. 
\end{proof}

\subsection{Realization of a collection of lines and (4,d)-nets}

\label{lines}
Which configuration of lines and points in $\mathbb P^2$ with given
incidence relation are possible? That is a classical question and even
for seemingly easy data the answer is often not clear. 

\begin{definition}
A $(4,d)$-net in $\PP^2$ is four collections by $d$ lines each of
them, such that exactly four lines pass through any point of intersection of two lines from
different collections, all these four lines are from different collections. 
\end{definition}

It is not clear whether a $(4,d)$-net exists for $d\geq 5$. In
\cite{2011arXiv1107.5530H} the authors proved, using tropical geometry, that there exists no
$(4,4)$-net.

One of the key ingredients is the following: if some net exists in the
classical world, then it exists in the tropical world. The problem appears: if
we have more than three tropical lines through a point on a plane,
then the intersection of two of them will be non-transversal.
However, thanks to modifications we always can have transversal
intersections, but probably in the space of bigger dimensions. For that,
we just do modifications along lines that have non-transversal
intersection. After these modifications, all intersections become transversal and the modified lines go to
infinity. Then, let us think about the following theorem, announced by the authors
of \cite{2011arXiv1107.5530H}, from the point of view of modifications:

\begin{problem} If some combinatorial data (required dimensions of  intersections of linear spaces) can be realized in
$\PP^k$ by a collection of linear spaces, does there exist a collection of tropical linear
spaces that realize the same combinatorial data in $\TT\PP^{k'}$ with
$k'\geq k$?
\end{problem}

Indeed, consider this realization in $\PP^k$. By passing to the tropical limit we obtain a tropical configuration, but the intersection dimensions may increase. Then, by doing the modifications, we want to repair the correct dimensions. Is it always possible to achieve?

\subsection{A point of big multiplicity on a planar curve}
\label{sec_multiplicitymodif}

In its most general form, this question could be formulated as
follows: given a cohomological class a of subvariety $S$ in a bigger variety, how many
singularities $S$ may have? For example, is it possible for a surface
of degree $4$ in $\CC P^4$ to have four double points and three two-fold
lines? 

There are several reasons why tropical geometry may provide tools
for such questions.
We will demonstrate these tools in the case of curves, where this deed has been already
done. Combinatorics of a planar tropical curve is encoded in the subdivision
of its Newton polygon. A singular point of multiplicity $m$
influences a part of the subdivision of area of order $m^2$ \cite{kalinin}, which is
in accordance with the order of the number of linear conditions $\big(\frac{m(m+1)}{2}\big)$ that a
point of multiplicity $m$ imposes on the coefficients of the curve's
equation.  For a general treatment of the tropical singularities, see 
\cite{kalinin}, \cite{nagatakalinin} and Chapters~1,2 in \cite{mythesis}.

In this section, we will only demonstrate how to apply modification techniques in this
problem, though we will obtain a weaker estimation -- but still of
order $m^2$.

The idea is the following: if a curve $C$ has a point $p$ of
multiplicity $m$, then for each curve $D$, passing through $p$, the
local intersection of $C$ and $D$ at $p$ is at least $m$.
The multiplicity of a local intersection of $C$ and $D$ can be estimated from above by studying
the connected component, containing $\Val(p)$, of
the stable intersection $\Trop(C)\cap \Trop(D)$ for the non-Archimedean amoebas of $C$ and $D$, see Theorem~\ref{th_stableintersection}.

Here is method: we take the polynomial $F$ defining $D$, and
use the fact that the image of $C$ under the map $m_D:(x,y)\to (x,y,F(x,y))$ intersects
the plane $z=0$ with multiplicity at least $m$. That
implies  {\bf existence} of a modification of $\Trop(C)$ along $\Trop(D)$, which has a leg of
weight $m$ going in the direction $(0,0,-1)$, exactly under the
point $\Val(p)$.
The latter modification is obtained just by  taking the non-Archimedean amoeba of $m_D(C)\subset m_D(\PP^2)$.

Now we reduce the problem to its combinatorial counterpart: is it
possible for two given tropical curves, that after the modification
along the second, the first curve will have a leg of weight
$m$, which projects exactly on the given point $\Val(p)$? After some work
with intrinsically tropical objects, we will get an estimate of this
point's influence on the Newton polygon of the curve.  

We are not going to consider this problem in the full generality, so
we will have a close look at the simplest interesting example. 
\begin{prop}
Suppose that a horizontal edge $E$ of a tropical curve $C$ contains a point $\Val(p)$ where $p$ is of multiplicity $m$ for a curve $C'$ such that $\trop(C')=C$. Denote by $d(E)$ the vertical edge in the dual subdivision of the Newton polygon which is
dual to $E$. Let the endpoints of $E$ be $A_1,A_2$ and two faces
$d(A_1),d(A_2)$ adjacent to $d(E)$ have no other vertical edges.  
Then the sum of widths of the faces $d(A_1),d(A_2)$ is at least $m$, so their
total area is at least $m^2/2$.
\end{prop}

\begin{proof} 
Suppose that $p$ is of multiplicity $m$ for $C'$. Let us take a line $D$
through $p$, such that $\Trop(D)$ contains inside
its vertical edge the point $\Val(p)$ . Clearly the local intersection $\Trop(C')\cap \Trop(D)$ is
one point, and the multiplicity of this point should be at least $m$.  That immediately
implies that the weight of $E$ is at least $m$. Hence the lattice
length of $d(E)$ is at least $m$.

Let us look at the dual picture in the Newton polygon. Two faces $d(A_1),d(A_2)$
adjacent to the vertical edge have the sum of width in the $(1,0)$
direction at least $m$ (by Proposition~\ref{prop_stableint}), $d(E)$ has length $m$, so the sum of the
areas of $d(A_1),d(A_2)$ is at least $m^2/2$.  
\end{proof}

\begin{remark}
Note that if the stable intersection of $\Trop(C)$ with the horizontal line is $m$, then we can uniquely determine the position of the valuation of the singular point, see Lemma~\ref{lem_positionsingular}.
\end{remark}

What to do if there is a usual horizontal line $L$, a part of $C$, through $\Val(p)$? We perform
the modification along this horizontal line $L$. If a part of the curve
goes to the minus infinity, this indicates that, along the corresponding initial degeneration, the equation
of \(D\) appears as a factor. If \(D\) is an actual component of \(C'\), then one
can divide the equation of \(C'\) by the equation of \(D\). That means that the Newton polygon of $C$
has two parallel vertical sides. The components of the modification which do not go to the minus infinity
do not contribute to the singularity.

However, it is possible that $d(A_1),d(A_2)$ have other vertical sides besides $d(E)$.
Let $\E$ be the stable intersection of $\Trop(C)$ and the
horizontal line; clearly $E\subset \E$.  Now, let us compute the sum of the areas of the faces $d(V)$ corresponding to vertices $V$ of $\Trop(C)$ on $\E$. It is possible that more
than two faces correspond to one singular point, if the edge with the singular point has an extension, see again Example~\ref{ex_singularexample}.

Suppose that a tropical curve has edges $A_1A_2,A_2A_3,\dots,A_{k-1}A_k$ and
$A_1,A_2,\dots,A_k$ are situated on a horizontal interval
$A_1A_k=\E$. Suppose that $p$, point of multiplicity $m$, is on the edge
$A_sA_{s+1}$. Making a modification along {\bf a line} containing $A_1A_k$ in
its horizontal ray we estimate only the common width of faces
corresponding to $A_1,A_2,\dots A_k$, which gives no good estimate for the
sum of areas of $d(A_i)$.

But we can make a modification along a quadric.

\begin{lemma}
In the above hypothesis, the sum of areas of all faces
$$d(A_1),d(A_2),\dots, d(A_k)$$ is at least $m^2/4$.
\end{lemma}

\begin{proof}  Let $a_i=\omega_{(1,0)}(d(A_{s+i})),i\geq 1$ be the width of $i$-th face (i.e. $d(A_{s+i})$) 
on the right, $b_i=\omega_{(1,0)}(d(A_{s-i})),i\geq 0$ be
the width of  $i$-th face (i.e. $d(A_{s-i})$) on the left. Let $c_i$ be the lattice length of $i$-th
vertical edge on the right (i.e. $c_i=\omega_{(0,1)}d(A_{s+i}A_{s+i+1}),i\geq 1$), $d_i$ be the length of the $i$-th vertical
edge on the left (i.e. $d_i=\omega_{(0,1)}d(A_{s-i}A_{s-i+1}),i\geq 1$). Then, let $\sum\limits_{i=1}^k a_i = A_k, \sum\limits_{i=1}^k b_i=B_k$. With the same calculations as above, making the modification along a piece of a quadric with vertices on $A_{s-j}A_{s+1-j}$ and $A_{s+i}A_{s+1+i}$
we get $A_i+c_i+B_j+d_j\geq m$ for all pairs $i,j$. Denote
\[
A=\min_i(c_i+A_i),\qquad B=\min_j(d_j+B_j).
\], so $ A+B\geq m$.

Then, $c_i\geq A-A_i,d_j\geq B-B_j$. Sum $S$ of areas can be estimated as $$2S\geq (m+c_1)A_1+\sum
(A_{i+1}-A_i)(c_i+c_{i+1})+(m+d_1)B_1+\sum
(B_{i+1}-B_i)(d_i+d_{i+1})$$

$$2S\geq (m+A-A_1)A_1+\sum (A_{i+1}-A_i)(A-A_i+A-A_{i+1})+$$
$$(m+B-B_1)B_1+\sum (B_{i+1}-B_i)(B-B_i+B-B_{i+1})\geq$$
$$A_1(m-A)+A^2+B_1(m-B)+B^2\geq m^2/2.$$

So, $S\geq m^2/4$.

\end{proof}

\section{Intuitions and interpretations}
\label{sec_modmotivation}

\rightline{\it La science toujours progresse et jamais ne faillit,} 
\rightline{\it toujours se hausse et jamais ne d\'eg\'en\`ere,}
\rightline{\it  toujours d\'evoile et jamais n'occulte.}
\rightline{\it Anonyme.}

This section explains why tropical modification is a natural notion and gives several interpretations in different contexts.
A reader primarily interested in definitions, examples, and theorems may proceed directly to the previous sections and return here only for motivation and references.

Tropical modifications were introduced in the seminal paper
\cite{mikh1} as the main ingredient in the tropical
equivalence relation. Namely, two tropical varieties are {\it equivalent} (tropical counterpart of birational isomorphism)
if they are related by a chain of tropical modifications and reverse
operations. For the full definition of an abstract tropical
  variety, refer to \cite{mikhalkin2006tropical} and \cite{MR2404949}.

The underlying idea is as follows. Recall that a tropical variety
$V$ can be decomposed into
a disjoint union of a compact part $V_c$ and a non-compact part
$V_\infty$, and $V=V_c\cup V_\infty$. Moreover, $V$ retracts on $V_c$. Then, the set
\(V_\infty\) consists of ``tree-like'' unions of pieces of affine hyperplanes.
We call these parts {\it legs} in the one-dimensional
case and {\it leaves} in general situation. For tropical curves, $V_\infty$ is a union of half-lines. For example, for a tropical elliptic
curve (see Figure~\ref{elliptic}, left side) the set $V_c$ is the ellipse, and $V_\infty$ is the set of trees growing on the ellipse. 

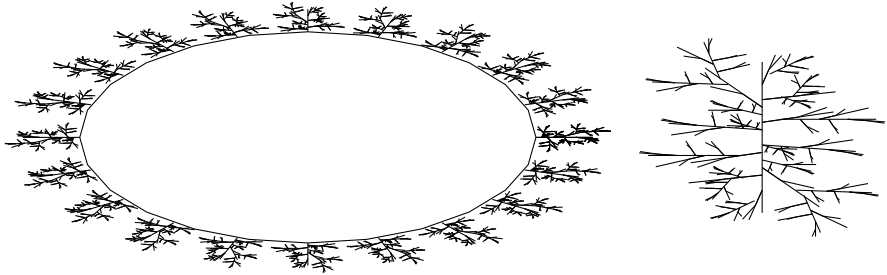
\begin{figure}[h]
\begin{center}
\begin{tikzpicture}[xscale=0.5, yscale=0.2]

\newcommand{\tree}[3]{%
\begin{scope}[shift={#1},rotate=#2,scale=#3]
\draw (0,0)--++(0,5);
\draw (0,1)--++(20:2);
\draw (0,1)--++(20:1)--++(40:1);
\draw (0,2)--++(120:3);
\draw (0,2)--++(120:1)--++(150:1);

\draw (0,2.5)--++(15:2);
\draw (0,2.5)--++(15:1)--++(25:0.5);
\draw (0,3.5)--++(80:2);
\draw (0,3.5)--++(80:1)--++(70:1);
\end{scope}
}%

\newcommand{\bigtree}[3]{%
\begin{scope}[shift={#1},rotate=#2,scale=#3]

\tree{(0,0)}{0}{1}

\tree{(0,0.5)}{80}{0.5}

\tree{(0.8,2.7)}{-20}{0.3}

\tree{(-1.1,3.85)}{10}{0.5}

\tree{(0,0.5)}{55}{0.2}

\tree{(0,1.5)}{60}{0.3}

\tree{(0.2,4.2)}{-40}{0.3}

\tree{(0.5,1.2)}{-90}{0.5}

\tree{(-0.9,3.5)}{85}{0.4}

\tree{(0.5,2.7)}{-130}{0.2}
\end{scope}
}

\begin{scope}[xshift=340]{
\bigtree{(0,0)}{0}{1}
\bigtree{(0,0)}{180}{1}}
\end{scope}

\foreach \p in {0,18,...,360}
{\bigtree{( 6*cos \p, 7*sin \p )}{\p-90}{0.3}} 

\draw plot [domain=-180:180] ( 6*cos \x, 7*sin \x );

\end{tikzpicture}
\end{center}
\caption{On the left side we see a tropical elliptic
  curve $V$ which is a part of the analytification of an elliptic
  curve. The ellipse is $V_c$ and the union of tree-like pieces is
  $V_\infty$. On the right side we see a tropical rational curve
  $V$, which is equal to $V_\infty$.  Each point $x$ of $V$ can serve as $V_c$, because $V$
  contracts onto any of its point $x\in V$.}
\label{elliptic}
\end{figure}

\begin{remark}On a tropical rational\footnote{Rational
    tropical varieties are the contractible ones, as a topological space. They are not well
    studied even in small dimensions. For example, there exist algebraic three-dimensional cubic
    hypersurfaces which are not rational. It is not known whether we can see
  this tropically, because all tropical cubic surfaces are contractible.}
variety $V$, each point may be chosen as $V_c$, see Figure~\ref{elliptic} right side. 
\end{remark}

%

Consider the tropical limit $V$ of algebraic varieties $W_{i}\subset
(\CC^*)^n$, i.e. $V=\lim_{i\to\infty}\Log_{t_i}(W_{i})$, where we
apply the map $\Log_{t_i}:\CC^*\to \RR, x\to \log_{t_i}|x|$
coordinate-wise and $\{t_i\}_{i=1}^\infty$ is a sequence of positive numbers, tending to $+\infty$. In this case the set $V_\infty$ encodes the topological
way of how $W_i$ approach some compactification of $(\CC^*)^n$. For the moment, the particular choice of
the compactification does not matter\footnote{For a fixed
  compactification, see the notion of sedentarity in
  \cite{shaw2013tropical} and \cite{BIMS}, p.~44. }.  

Besides, for $i$ big enough, the Bergman fan
$B(W_i):=\lim_{t\to\infty}\Log_t(W_i)$ of $W_i$ 
is equal to $\lim_{t\to\infty}\frac{1}{t}V$. The latter limit is obtained by
contracting the compact part $V_c$ of $V$, so the Bergman fan can be
restored by $V_\infty$. Note, that $V$ came here with a particular
immersion to $\RR^n$. 

\begin{example}
If curves $W_i, i=1,2,\dots$ in $(\CC^*)^2$ all have branches with
asymptotic $(s^k,s^l)$ with a local parameter $s\to \infty$, then the
tropical limit $V$ of this family lies in
$\RR^2$, and $V$ has the infinite leg (half-line) in the lattice direction
$(k,l)$. 
\end{example}

Let us suppose that we have an algebraic map $f:(\CC^*)^n \to
(\CC^*)^m$, and $f$ is in general position with
respect to the family $\{W_i\}$, i.e. for each $i$ big enough, the
image $f(W_i)$ is birationally equivalent to $W_i$.  Let $V'$ be the
tropical limit  of the family $\{f(W_i)\}$. One can prove that $V'_\infty$ differs from
$V_\infty$ by adding new half-planes and contracting other
half-planes. These half-planes
grow along the tropicalization of zeros and poles of $f$ on
$W_i$ (exactly as in Definition~\ref{def_modification}). This consideration suggests the
ideas of {\it modification} and, subsequently, {\it tropical equivalence}. The name ``modification'' was borrowed from complex analysis, and tropical modification is sometimes called ``tropical blow-up''.  

In Section \ref{category} we see how the notion of modifications
allows us to define the {\it category of tropical curves}. This category keeps
track of birational isomorphism in the category of complex algebraic
curves.  See also \S\ref{weil}, where
making modifications for curves simplifies a proof to some extent.

Alternatively, tropical geometry can be thought of as the study of skeletons of analytifications of algebraic varieties; see Figure~\ref{elliptic}, where the analytification of an elliptic curve is shown on the left and the analytification of $\PP^1$ on the right. The analytification $X^{an}$ of a variety $X$ is the
set of all seminorms on functions on $X$. Each point $x\in X$ defines
such a seminorm by measuring the order of vanishing of a function at
$x$.
In Figure~\ref{elliptic}, these points are represented by the ends of leaves; equivalently, these valuations represent norms of ``zero'' radius. The analytification of an elliptic curve is the injective limit of all modifications of its tropicalization: we add a leg at every point of a circle, then add a leg at every point of the new space, and so on.

For the sake of shortness, we refer the reader to a nice introduction
in Berkovich spaces, with a bit of pictures
\cite{baker2008introduction},\cite{2010arXiv1010.2235T} and to \cite{2011arXiv1104.0320B} to see
how it has been applied to tropical geometry (also, see on the page 7 in \cite{2011arXiv1104.0320B},
using of $\log$ reminds hyperbolic approach).

 We can obtain a tropical
variety $V$ as the non-Archimedean amoeba of an algebraic variety $W$
over a non-Archimedean field.  This approach (see section \S\ref{berkovich}) finally
suggests the same idea of equivalence up to modification, because the
analytification $W^{an}$ is the injective limit of all ``affine'' tropical modifications (i.e. along only principal
divisors) of $V$ (see \cite{zbMATH05608662}). Berkovich proved that
$W^{an}$ retracts on a finite polyhedral complex, so
$V_c$ is a deformation retract of $W^{an}$. Even better, the metric on $W^{an}$
agrees with the metric on $V$ for the case of curves\footnote{That
  should be true for varieties of any dimension, modulo integer affine
transformations, but no proof has appeared yet. For the skeletons in higher dimensions see \cite{Gubler:2014eu, Gubler:2015pd}.}
(\cite{2011arXiv1104.0320B}). For elliptic curves $V_c$ will be a circle in both tropical
and analytical cases, and its length is prescribed by the $j$-invariant of
the considered curve (\cite{2014arXiv1409.7430A}).

This connection between tropical geometry and analytic geometry leads to the
questions of {\it lifting} or {\it realizability}, i.e. what could be
the intersection of two varieties $X,Y$ if we know the intersection of their tropicalizations?
If their tropicalizations $\Trop(X),\Trop(Y)$ intersect
transversally, the answer is relatively simple, see \cite{2010arXiv1007.1314O}. If
the intersection of $\Trop(X),\Trop(Y)$ is non-transverse, then we can lift {\it the stable
intersection} of these tropical varieties, see
\cite{2011arXiv1109.5733O},\cite{MR2900439}.  
 
This raised the following question:  to what extent the only condition for a
divisor on a curve to be realizable as an intersection is to be
rationally equivalent to the stable intersection (cf. \cite{Morrison:2014mz}, Conjecture 3.4)?

Tropical modification (as a {\it method}) helps dealing with such questions. It is known that being rationally equivalent to the stable intersection is not enough.  We consider other existing
obstructions (in fact, equivalent to Vieta theorem) for what can happen in non-transverse
tropical intersections, and prove, for that occasion, the tropical
Weil reciprocity law by using the tropical momentum Lemma~\ref{lem_moment1}.

Consequently, modifications are used in tropical
intersection theory (\cite{Shaw:2015qd,shaw}), to
define the intersection product. Nevertheless, one must use modifications along non-Cartier divisors (Examples 1.1.37,
3.4.18 in \cite{shaw}, for moduli space of five points on rational curve) and even along
non-realizable subvarieties -- for a proof that they are
non-realizable as tropical limits.

As we stated before, one should think that a tropical modification
along $X$ reveals asymptotical behavior of objects near $X$. We can find an analogy
in non-standard analysis: the tropical line is the hyperreal line, the modification at a point is an
approaching this point with {\it an infinitesimal telescope}, see
Figure \ref{telescope} and Section \ref{berkovich}. In order to define tropical Hopf manifolds one should also use the modifications to study certain germs \cite{ruggiero2015tropical}.

Given a surface with hyperbolic structure, we can make a puncture at $x$. This
changes the hyperbolic structure and $x$ goes, in a sense, to
``infinity''. A tropical curve can be obtained as a degeneration of
hyperbolic structures, and making a puncture at $x$ results as the
modification at the limit of $x$, see Section \S\ref{sec_hyperbolic}. 

A modification can be described as a graph of a function, if we
use the convention about multivalued addition, brought in
tropical geometry by Oleg Viro (\cite{viro}), see Section~\ref{sec_moddefinition}.

The other applications of tropical modification as a {\it method} are following.
Passing to tropical limit squashes a variety, and some local features
become invisible. In order to reveal them back we can do a
modification (whence also this metaphor ``look in an infinitesimal
  microscope'').
For example, modifications allow us to restore transversality between
lines if we have lost it
during tropicalization (\S\ref{lines}), then it allows us to see (-1)-curves on del
Pezzo surfaces (\cite{2014arXiv1402.5651R}).  Methods of lifting
non-transverse intersections leads us to use modifications in questions about singularities:
inflection points -- \cite{brugalle}, singular points --
\cite{markwig}. As an example (Section~\ref{sec_multiplicitymodif}), we use modification in the
study of singular points of order $m$ (but obtain weaker results than
in \cite{kalinin}).

\subsection{Hyperbolic approach and moduli spaces}
\label{sec_hyperbolic}

Consider a tropical curve $C$ given as the tropical limit of complex curves
$C_i$. From the point of view of hyperbolic geometry, a modification of $C$ at a
point $x\in C$ amounts to creating a puncture \(x_i\) on \(C_i\), with the condition that \(x_i\to x\). To explain this we need to know how to directly construct
tropical curves via limits of Riemann surfaces with hyperbolic
structure on them, without any immersions\footnote{Usually people
  consider curves $C_i$ in toric variety $X$ and then they consider degeneration
  of complex structures on $X$.}.

For details on how tropical geometry can be built on the basis of hyperbolic geometry, see \cite{lionel}. Here we briefly sketch the construction.

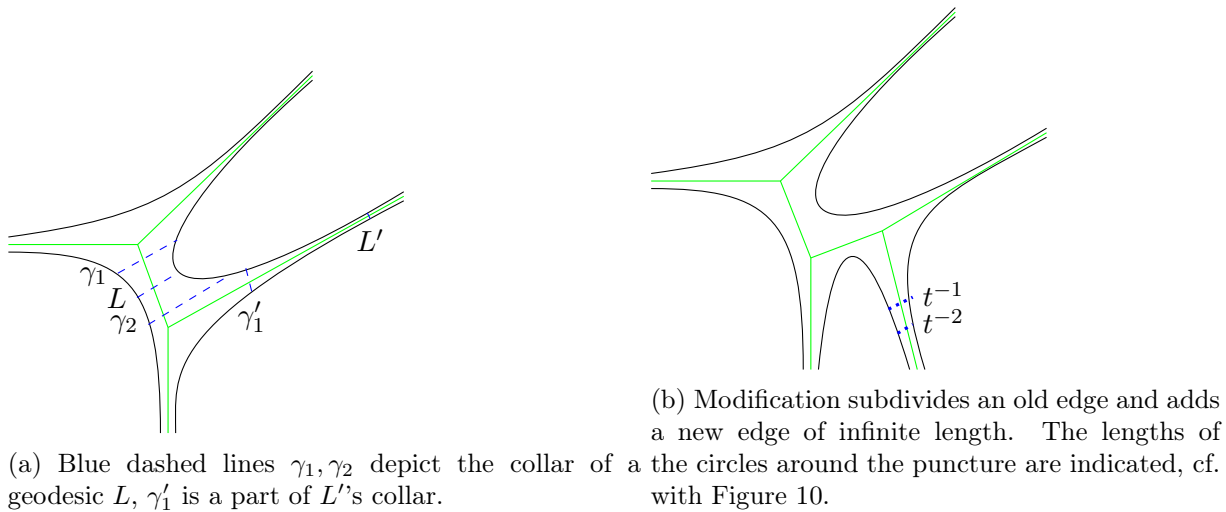
\begin{figure}[htbp]

\begin{center}
\begin{subfigure}[b]{0.5\textwidth}
\begin{tikzpicture}[scale=0.4]

\begin{scope}[xshift=-2cm]
\draw (0,0) .. controls (0,5) and (-1,6).. (-5,6);
\draw (-5,6.5) .. controls (0,7.2) and (1,8).. (5,12);
\draw (5,11.7) .. controls (-2,5.5) and (-1,2.5).. (8,8);
\draw (8,7.7) .. controls (0,3.5) and (0.5,2).. (0.5,0);

\newcommand{\p}{-0.75}

\draw[green] (-5,6.25)--(\p,6.25)--(0.25,3.5)--(0.25,0);
\draw[green] (\p,6.25)--(5,11.85);
\draw[green] (0.25,3.5)--(8,7.85);
\draw[dashed,blue] (-0.77,4.49)--(0.55,5.3);
\draw[dashed,blue] (-1.4,5.3)--(0.55,6.4);
\draw[dashed,blue] (-0.4,3.6)--(2.3,5.25);
\draw (-0.77,4.49) node[left]{$L$};
\draw (-1.4,5.2) node[left] {$\gamma_1$};
\draw (-0.4,3.6) node[left]{$\gamma_2$};

\draw[dashed, blue] (6.9,7.1)--(6.8,7.3);
\draw[dashed, blue] (3,4.7)--(2.8,5.5);
\draw(6.9,7.1) node[below]{$L'$};
\draw(3,4.7) node[below]{$\gamma_1'$};
\end{scope}
\end{tikzpicture}
\caption{Blue dashed lines $\gamma_1,\gamma_2$ depict the collar of
  a geodesic $L$, $\gamma_1'$ is a part of $L'$'s collar.}
\end{subfigure}
\begin{subfigure}[b]{0.45\textwidth}
\begin{tikzpicture}[scale=0.4]
\draw (0,0) .. controls (0,5) and (-1,6).. (-5,6);
\draw (-5,6.5) .. controls (0,7.2) and (1,8).. (5,12);
\draw (5,11.7) .. controls (-2,5.5) and (-1,2.5).. (8,8);

\draw (0.5,0) .. controls (1,5) and (2,5).. (3.5,0);
\draw (8,7.7) .. controls (4,5.5) and (2.5,5).. (4,0);

\newcommand{\p}{-0.75}

\draw[green] (-5,6.25)--(\p,6.25)--(0.25,3.7)--(0.25,0);
\draw[green] (\p,6.25)--(5,11.85);
\draw[green] (0.25,3.7)--(2.6,4.6)--(8,7.85);
\draw[green](2.6,4.6)--(3.75,0);

\draw[dotted, blue, very thick] (2.8,2)--(3.6,2.4);
\draw (3.6,2.5) node[right]{$t^{-1}$};
\draw[dotted, very thick, blue] (3.1,1.2)--(3.6,1.5);
\draw (3.6,1.6) node[right]{$t^{-2}$};

\end{tikzpicture}
\caption{Modification subdivides an old edge and adds a new edge of infinite length. The lengths of the circles around the puncture are indicated, cf. with Figure~\ref{telescope}. }
\end{subfigure}
\end{center}
\caption{We draw the tropical limits of Riemann surfaces, and a surface close to the limit. Modification adds a puncture to each curve in the family and a leg to the tropical curve.}
\label{collar}
\end{figure}

The approach proposed by L. Lang uses the collar lemma \cite{Buser:1978kx}. This lemma says that any closed geodesic of length $l$ has a collar of width $\log(\coth(l/4))$ and, more importantly, that collars of distinct sufficiently short closed geodesics do not intersect; see Figure~\ref{collar}. It is also important that shorter geodesics have wider collars, and, intuitively, a puncture has a collar of infinite width.

Thus, given a family of curves $C_i$ (of the same genus), we consider a fixed 
pair-of-pants decomposition by geodesics $L_i$. The tropical curve is constructed as follows: its vertices are in one-to-one correspondence with the pair-of-pants, each shared boundary component between two pairs-of-pants correspond to an edge of the tropical curve, and the collar lemma furnishes us with the length of the edges of the tropical curve as the
logarithms (with base $t$, and $t\to\infty$ as the hyperbolic
structure degenerates) of widths of the collars of $L_i$'s. Compare this approach 
with \cite{Bowditch:1988rm}.

What will happen if we make a puncture? A puncture is the limit of
small geodesic circles. Cutting out a disk with radius $t^{-n}$ adds a leaf of finite
length $n$, as it is seen from the above description. Therefore, cutting
out a point results in adding an infinite edge, i.e. a modification.

That explains why a permanent using of graphs for moduli
space problems is actually useful (\cite{kontsevich1992}, cf. \cite{2007arXiv0709.3953K}). Tropical curves describe the part of
boundary of a moduli space, and modification corresponds to marking
a point (read \cite{do2008intersection} to see the hyperbolic view on
moduli space problems), which are punctures from the hyperbolic point of view (see applications to moduli space of points \cite{MR2404949}).
Tropical differential forms are also defined in this manner while taking a limit of
hyperbolic structure \cite{mikhalkin2006tropical}.

\subsection{Non-standard analysis}
\label{berkovich}
Non-standard analysis appeared as an attempt to formalize the notion of
``infinitesimally small'' variables (see \S4 of \cite{MR3026767} for a nice
and short exposition).

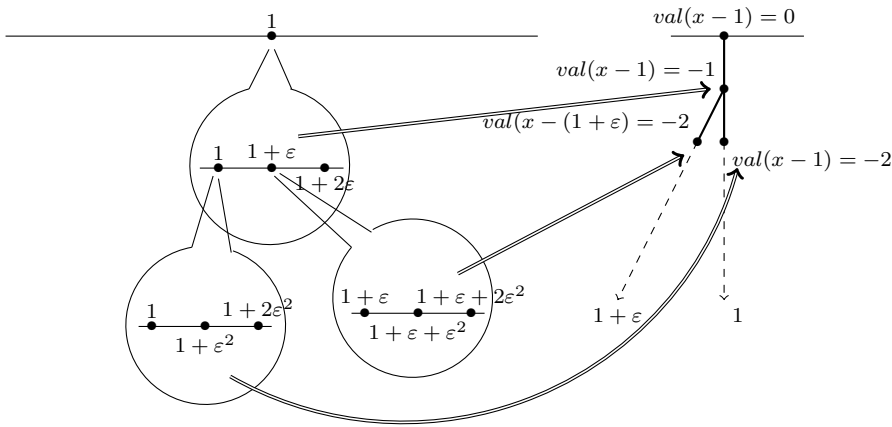
\begin{figure}[h]
\begin{tikzpicture}[scale=0.35]
{\scriptsize
\draw (-10,10)--(10,10);
\node at (0,10) {$\bullet$};
\draw (0,10) node[above]{$1$};

\draw (-0.85,8) arc (105:435:3);
\draw (-0.85,8)--(-0.1,9.5);
\draw (0.75,8)--(0.1,9.5);

\draw (-2.7,5)--(2.7,5);
\node at (2,5) {$\bullet$};
\node at (0,5) {$\bullet$};
\node at (-2,5) {$\bullet$};

\draw (-2,5) node [above]{$1$};
\draw (0,5) node [above]{$1+\varepsilon$};
\draw (2,5) node [below]{$1+2\varepsilon$};

\draw (-3,2) arc (100:430:3);
\draw(-3,2)--(-2.2,4.7);
\draw (-1.52,1.8)--(-2,4.6);
\draw (-5,-1)--(0,-1);
\node at (-4.5,-1) {$\bullet$};
\draw (-4.5,-1) node [above]{$1$};
\node at (-2.5,-1) {$\bullet$};
\draw (-2.5,-1) node [below]{$1+\varepsilon^2$};
\node at (-0.5,-1) {$\bullet$};
\draw (-0.5,-1) node [above]{$1+2\varepsilon^2$};

\draw (3,2) arc (140:480:3);
\draw (3,2)--(0.1,4.7);
\draw(3.75,2.57)--(0.3,4.8);
\draw (3,-0.5)--(8,-0.5);
\node at (3.5,-0.5) {$\bullet$};
\draw (3.5,-0.5) node [above]{$1+\varepsilon$};
\node at (5.5,-0.5) {$\bullet$};
\draw (5.5,-0.5) node [below]{$1+\varepsilon +\varepsilon^2$};
\node at (7.5,-0.5) {$\bullet$};
\draw (7.5,-0.5) node [above]{$1+\varepsilon +2\varepsilon^2$};

\draw (15,10)--(20,10);
\node at (17,10) {$\bullet$};
\draw (17,10) node[above]{$val(x-1)=0$};

\draw[dashed,->] (17,6)--(17,0);
\draw (17,0) node[below right]{$1$};
\draw[thick](17,10)--(17,6);
\node at (17,8) {$\bullet$}; 
\draw (17,8) node[above left]{$val(x-1)=-1$};
\node at (17,6) {$\bullet$}; 
\draw (17,6) node[below right]{$val(x-1)=-2$};

\draw[thick] (17,8)--(16,6);
\node at (16,6) {$\bullet$}; 

\draw (16,6) node[above left]{$val(x-(1+\varepsilon)=-2$};
\draw[dashed,->] (16,6)--(13,0);
\draw (13,0) node[below]{$1+\varepsilon$};

\draw[thin, double,<-] (16.5,8) -- (1,6.2);
\draw[thin, double,<-] (15.5,5.5) -- (7,1);
\draw[thin, double,<-] (17.5,5) arc (345:240:13);
}
\end{tikzpicture}
\caption{Similarity in the pictures while using an infinitesimal microscope (left)
and the tropical modification at points 1 and $1+\varepsilon$ (right).}
\label{telescope}
\end{figure}

There is a way to understand tropical geometry via nonstandard analysis
(cf. \S1.4 \cite{oberwolfach}). Figure~\ref{telescope} shows that
tropical modifications are similar to ``infinitesimal microscope'' for  the hyperreal line in
the terminology of \cite{keisler1976foundations}, and this
interpretation in computational sense is the same as for Berkovich
spaces: doing modification at the point $x=1$ on a curve is adding a
leg to the tropical curve, which ranges points according their
asymptotical distance to $x=1$, i.e. $\val(x-1)$, these pictures are also
similar to the hyperbolic ones; see Figure~\ref{collar}. Dotted lines represent directions to the endpoints of the analytifications, and we have a similar type of branching at all points in Figure~\ref{elliptic}.

It is worth to note that there are still no applications of this point
of view, neither in tropical geometry, nor in non-standard
analysis. However, Berkovich spaces can be understood as a modern version of
non-standard analysis, and tropical modification has
applications there.

\section{Questions}
An important feature of tropical geometry is that it builds a bridge between geometric objects, such as hyperbolic surfaces, and discrete objects, such as $p$-adic valuations and non-Archimedean analytic spaces. Since tropical modifications dwell in both realms, we expect them to be useful in future applications.

\begin{problem}
Explore tropical modifications over the phase-tropical numbers \cite{viro}.
\end{problem}
For example, in \cite{el2018constructing} tropical geometry was used to construct a system of polynomial equations with only  5 monomials and 6 positive solutions, and the authors used non-transverse intersections. Looking at questions of such type with a modification tool over phase tropical numbers looks promising.

\begin{problem}
Study the tangency between tropical curves and tropical surfaces in $\TT^3$. 
\end{problem}
For example, one can ask how many surfaces of a given degree are tangent to a given collection of lines (or twisted cubics) in $\TT^3$.


\end{document}